\documentclass{svmult}
\pdfoutput=1
\usepackage{ae,aecompl}
\usepackage{amssymb,amsmath} % non compatibile con AMSthm
\usepackage[english]{babel}

\allowdisplaybreaks[4]

\newcommand{\cvd}{\hfill\ensuremath{\blacksquare}}

\newtheorem{thm}{Theorem}[section]
\newtheorem{lem}[thm]{Lemma}
\newtheorem{prop}[thm]{Proposition}
\newtheorem{cor}[thm]{Corollary}
\newtheorem{df}[thm]{Definition}

\newtheorem{rem}[thm]{Remark}

\newcommand{\hO}{\mathcal{O}}
\newcommand{\A}{\mathcal{A}}
\newcommand{\B}{\mathcal{B}}
\renewcommand{\E}{\mathcal{E}}

\newcommand{\OO}{\mathcal{A}}%\mathcal{O}
\newcommand{\HH}{\mathcal{H}}
\newcommand{\U}{\mathcal{U}}
\newcommand{\Uq}[2]{U_q(\mathfrak{#1}(#2))}
\newcommand{\R}{\mathbb{R}}
\newcommand{\C}{\mathbb{C}}
\newcommand{\CP}{\mathbb{C}\mathrm{P}}
\newcommand{\N}{\mathbb{N}}
\newcommand{\Z}{\mathbb{Z}}
\newcommand{\Q}{\mathbb{Q}}
\newcommand{\mL}{\mathcal{L}}
\newcommand{\de}{\partial}
\newcommand{\deb}{\bar{\partial}}
\newcommand{\dd}{\mathrm{d}}
\newcommand{\inner}[1]{\left<#1\right>}
\newcommand{\mc}[1]{\mathcal{#1}}
\renewcommand{\bar}{\overline}
\newcommand{\tr}{\mathrm{Tr}}
\newcommand{\id}{\mathtt{id}}
\newcommand{\ket}[1]{\left|#1\right>}
\newcommand{\ma}[1]{\bigg(\begin{array}{cc}#1\end{array}\bigg)}
\newcommand{\az}{\triangleright}
\newcommand{\za}{\triangleleft}
\newcommand{\wprod}{\wedge_q\mkern-1mu}
\newcommand{\dwprod}{\,\dot\wedge_q}

\newcommand{\dotimes}{\,\dot{\otimes}\,}

\newcommand{\vol}{\mathbf{\Phi}}
\newcommand{\kahl}{\mathbf{\Omega}_q}

\begin{document}

\title*{Geometry of Quantum Projective Spaces}

\author{Francesco D'Andrea\inst{1}\and Giovanni Landi\inst{2}}
 
\institute{%
Dipartimento di Matematica e Applicazioni,
Universit\`a di Napoli Federico II, P.le Tecchio 80, I-80125 
Napoli, Italy \\[2pt]
\textit{Email address:} \texttt{francesco.dandrea@unina.it}
\vspace{5pt}\and
Dipartimento di Matematica, Universit\`a di Trieste,
Via A.~Valerio 12/1, I-34127 Trieste, Italy, and INFN, Sezione di Trieste, Trieste, Italy \\[2pt]
\textit{Email address:} \texttt{landi@units.it}}
\maketitle

\bigskip
\noindent {\bf Abridged table of contents}: \\ ~\\
\S\ref{intro} Introduction \\
\S\ref{qsc} The quantum $SU(n+1)$ and $\CP^n$ \\
\S\ref{sec:3} K-theory and K-homology \\
\S\ref{sec:equivST} Dirac operators and spectral triples \\
\S\ref{plncm} The projective line $\CP^1_q$ as a noncommutative manifold \\
\S\ref{sec:Cal} A digression: calculi and connections \\
\S\ref{sec:5} The complex structure of $\CP^n_q$ \\
\S\ref{sec:6} Monopoles and instantons on $\CP^2_q$ \\
\S\ref{app} On Chern characters and Fredholm modules 

\vfill
\noindent {\bf Acknowledgements}.
F.D. was partially supported by the italian ``Progetto FARO 2010'' and by the ``Progetto Giovani GNFM 2011''
(INDAM, Italy). G.L. was partially supported by the Italian Project ``Cofin08 -- Noncommutative Geometry,
Quantum Groups and Applications".  %GL thanks

\newpage

%%% ======================================================================

\section{Introduction}\label{intro}
In recent years, several quantizations of real manifolds
have been studied, in particular from the point of view of Connes' noncommutative
geometry \cite{Con94}. Less is known for complex noncommutative
spaces. A natural first step in developing a theory is clearly the study
of quantizations of flag manifolds --- and, in particular, complex
projective spaces ---, thus explaining the increasing
interest for this class of examples.
In this paper, we review various aspects of the geometry of
deformations of complex projective spaces.

Some references on these topics are the following.
For Fredholm modules and classical characteristic classes,
as well as equivariant K-theory and quantum characteristic classes,
one can see \cite{HS02,DL09a,DL09b};
differential calculi have been studied by several authors,
e.g.~\cite{CHZ96,Wel00,Kra04,HK04,HK06b,Bua11a,Bua11b};
for Dirac operators and spectral triples we refer to
\cite{DS03,Kra04,DDL08b,DD09,DL09a};
complex structures and positive cyclic cocycles have been
studied in \cite{KLS10,KM11a,KM11b};
for monopoles and instantons in the $4$-dimensional case,
we refer to \cite{DL09b,DL11}.
The quantum projective line has been also used as
the ``internal space'' for a scheme of equivariant dimensional
reduction leading to $q$-deformations of systems of non-abelian
vortices in \cite{LS10}.
In the (complex) $1$-dimensional case, in \cite{Wag09} there is a study of some of
the ``seven axioms'' of noncommutative geometry.

An original part of the present work is a proof of rational Poincar\'e
duality for a new family of real spectral triples, generalizing
the one in \cite{Wag09}.
Other original results include: the computation in \S\ref{sec:hodge} of the cohomology
of the Dolbeault complex of $\CP^n_q$; in \S\ref{se:cpnhs} we give in full details an easier and complete
proof of a dimension formula (Cor.~4.2 of \cite{KM11b}) for the zero-th cohomology of holomorphic modules;
in \S\ref{sec:exist} we exhibit $n+1$ positive Hochschild twisted cocycles that, in \S\ref{sec:5.4}
we pair with equivariant K-theory, thus also showing the  pairwise inequivalence of the projections 
in Prop.~\ref{prop:psigma}.  

\smallskip
\noindent {\bf Notations}.
We shall have $0<q\leq 1$ as deformation parameter,
with $q=1$ corresponding to the ``classical limit''.
The $q$-analogue of an integer number $n$ is defined as
$ [n]: =(q^n-q^{-n})/(q-q^{-1})$
%\frac{q^n-q^{-n}}{q-q^{-1}} \;.
%
for $q\neq 1$ and equals $n$ in the limit $q\to 1$. 
For any $n\geq 1$, the $q$-factorial is 
$
[n]!:=[n][n-1]\ldots [1] %\;,
$,
with $[0]!:=1$, and, for $j_0,\ldots,j_n$ interger numbers, the $q$-multinomial coefficients is
$$
[j_0,\ldots,j_n]!:=\frac{[j_0+\ldots+j_n]!}{[j_0]!\ldots[j_n]!} \;.
$$
%The symbol $[a,b]_q$ denotes the $q$-commutator of two operators $a,b$:
%$$
%[a,b]_q=ab-q^{-1}ba \;.
%$$
By \mbox{$*$-algebra} we shall always mean an unital involutive associative algebra over
the complex numbers, and by representation of a \mbox{$*$-algebra} we always mean a unital $*$-representation,
unless otherwise stated. For a coproduct we use Sweedler notation,
$\Delta(x)=x_{(1)}\otimes x_{(2)}$, with summation understood.

%%% ======================================================================
\section{The quantum $SU(n+1)$ and $\CP^n$}\label{qsc}

\subsection{Quantized `coordinate rings'}\label{sec:2.2}
In the framework of $C^*$-algebras, the compact quantum groups $SU_q(n)$,
for $n\geq 2$, have been introduced in \cite{Wor88}. It is well known that any
compact quantum group has a dense subalgebra which is a Hopf \mbox{$*$-algebra}
with the induced coproduct. This Hopf \mbox{$*$-algebra} is the analogue of the algebra 
of representative functions of a compact group. For $SU_q(n)$ it will be denoted by $\OO(SU_q(n))$;
it has been studied in \cite{FRT88} (among others) and the exact definition can be found for 
example~in \cite[Sect.~9.2]{KS97}.
Here we recall that, for any $n_1>n_2$, there is a surjective
Hopf \mbox{$*$-algebra} morphism $\OO(SU_q(n_1))\to \OO(SU_q(n_2))$ 
or, in other words, $SU_q(n_2)$ is a ``quantum subgroup'' of $SU_q(n_1)$.

For $n\geq 1$, the ``quotient''  $SU_q(n+1)/SU_q(n)$ leads to the so-called
odd-dimensional quantum spheres $S^{2n+1}_q$, the natural ambient
space when studying quantum projective spaces. More precisely,
the coordinate algebra  $\OO(S^{2n+1}_q)$ is defined as the $*$-subalgebra
of $\OO(SU_q(n+1))$ made of coinvariant elements for the coaction
of $SU_q(n)$.
As an abstract \mbox{$*$-algebra}, this is generated by $2(n+1)$ 
elements $\{z_i,z_i^*\}_{i=0}^n$ with commutation relations \cite{VS91}:
\begin{gather*}
z_iz_j =q^{-1}z_jz_i  \;, \quad \;0\leq i<j\leq n \; \qquad \textup{and} \qquad 
z_i^*z_j =qz_jz_i^*  \;, \quad \;i\neq j \;, \\
[z_n^*,z_n]=0 \qquad \textup{and} \qquad 
[z_i^*,z_i] =(1-q^2)\sum\nolimits_{j=i+1}^n z_jz_j^* 
\;, \quad i=0,\ldots,n-1 \;, \\
\intertext{and sphere condition: }
z_0z_0^*+z_1z_1^* +\ldots+z_nz_n^*=1 \;.
\end{gather*}
Using the commutation relations above, an equivalent way to write the sphere condition is 
$\sum_{j=0}^n q^{2j}z_j^* z_j=1$. 
In the case $n=1$, one has $SU_q(1)=\{1\}$ and the quantum sphere $S^3_q$ is the `manifold'  
underlying the quantum 
$SU(2)$ group. The generator $z_0=\alpha$ and $z_1=\beta$ of the algebra $\OO(S^3_q)$ can be assembled into a matrix
$$
U:=\ma{\alpha & \beta \\ -q\beta^*\; & \alpha^*} \;.
$$
The defining relations are then encoded
in the condition $UU^*=U^*U=I_2$, where $I_2$ is the $2\times 2$ identity
matrix. With standard `matrix' coproduct, counit
and antipode, this gives the well known quantum group $SU_q(2)$ of 
\cite{Wor79,Wor87}.

The original notations of \cite{VS91} are obtained by setting $q=e^{\hbar/2}$;
the generators of \cite{DD09} correspond to the replacement $z_i\to z_{n+1-i}$,
while the generators $x_i$ of~\cite{HL04} are related to ours by
$x_i=z_{n+1-i}^*$ and by the replacement $q\to q^{-1}$.

For any $n\geq 1$, the $*$-subalgebra of $\OO(S^{2n+1}_q)$ generated by
$p_{ij}:=z_i^*z_j$ will be denoted $\OO(\CP^n_q)$, and called 
the algebra of `polynomial functions' on the quantum projective space $\CP^n_q$. 
The algebra $\OO(\CP^n_q)$ is made of invariant elements for the $U(1)$ action 
$ z_i  \to \lambda z_i $ for $\lambda \in U(1)$.

From the relations of $\OO(S^{2n+1}_q)$ one gets analogous quadratic
relations for $\OO(\CP^n_q)$ \cite{DL09a}.
In particular, the elements $p_{ij}$ are the matrix entries of a projection $P= (p_{ij})$, i.e. $P^2=P=P^*$ or 
$\sum_{j=0}^n p_{ij} p_{jk} = p_{ik}$ and $p_{ij}^*=p_{ji}$. 
This projection has  $q$-trace:
\begin{equation}\label{q-tr}
\tr_q(P):=\sum\nolimits_{i=0}^n \, q^{2i} p_{ii}=1.
\end{equation}
For $n=1$, $\CP^1_q$ is also a deformation of the unit sphere $S^2$
known as the ``standard'' Podle\'s quantum sphere \cite{Pod87}.

A further generalization is given by quantum weighted projective spaces
$\mathbb{WP}_q(k_0,k_1,\ldots,k_n)$, where $\{k_i\}_{i=0}^n$
are pairwise coprime numbers. The corresponding coordinate algebra is the
fixed point subalgebra of $\OO(S^{2n+1}_q)$ for the action $z_i\mapsto\lambda^{k_i}z_i$
of $U(1)$. For $n=1$, these are called ``quantum teardrops'' and studied
in \cite{BS11}. Their discussion is beyond the scope of this review.

\subsection{Symmetry algebras}\label{sec:2.1}
Let $\A$ be a \mbox{$*$-algebra}, $(\U,\epsilon,\Delta,S)$ a Hopf \mbox{$*$-algebra}, One says that $\A$ is a left $\U$-module \mbox{$*$-algebra} if there is a left action `$\az$' of $\U$ on $\A$ such
that
$$
x\az ab=(x_{(1)}\az a)(x_{(2)}\az b)\;,\quad
x\az 1=\varepsilon(x)1\;,\quad
x\az a^*=\left( S(x)^*\az a \right)^*\;,
$$
for all $x\in\U$, $a,b\in\A$. 
With this data, one defines the left crossed product algebra $\A\rtimes\U$. i.e.~the
$*$-algebra generated by $\A$ and $\U$ with crossed commutation relations
$xa=(x_{(1)}\az a)x_{(2)}$, for all $x\in\U$ and $a\in\A$. There are analogous notions of a 
right $\U$-module \mbox{$*$-algebra} and right crossed product algebra.

Symmetries of $SU_q(n+1)$ and of related quotient spaces are described by
the action of the dual Hopf \mbox{$*$-algebra}, here denoted by $\Uq{su}{n+1}$. 
This is the `compact' real form of the Hopf
algebra denoted $\breve{U}_q(\mathfrak{sl}(n+1,\C))$ in \S6.1.2 of \cite{KS97}.
Left and right canonical (and commuting) actions of $\Uq{su}{n+1}$ on $\OO(SU_q(n+1))$ will be denoted by $\az$ 
and $\za$ respectively.  

If $h:\OO(SU_q(n+1))\to\C$ is the Haar state,
i.e.~the unique invariant state on the algebra, a inner product  
is defined as usual by $\inner{a,b}:=h(a^*b)$. It turns out that the left
action is unitary for this inner product, that is
$\inner{a,x\az b}=\inner{x^*\az a,b}$ for all $a,b\in\OO(SU_q(n+1))$
and $x\in\Uq{su}{n+1}$. The right action is not unitary,
but it can be turned into a second unitary left action $\mL$,
commuting with the former one, via the rule
$$
\mL_xa:=a\za S^{-1}(x) \;.
$$
The algebra $\OO(S^{2n+1}_q)$ can be identified with the $*$-subalgebra
of $\OO(SU_q(n+1))$ fixed by the $\mL$-action of the Hopf $*$-subalgebra
$\Uq{su}{n}\subset\Uq{su}{n+1}$; whereas the projective space $\OO(\CP^n_q)$
is the $*$-subalgebra of $\OO(S^{2n+1}_q)$ fixed by a further $\mL$-action
of a classical Lie algebra $\mathfrak{u}(1)$.

As the two left actions commute, both $\OO(S^{2n+1}_q)$ and $\OO(\CP^n_q)$
are themselves  left $\Uq{su}{n+1}$-module \mbox{$*$-algebra}s for the action `$\az$'.

Let us give few more details for the $n=1$ case needed later on,
while we refer to the literature for the $n>1$ case.
The Hopf \mbox{$*$-algebra} $\Uq{su}{2}$ is generated by
$K=K^*,K^{-1},E,F=E^*$ with relations
$$
KEK^{-1}=qE\;,\qquad
[E,F]=(K^2-K^{-2})/(q-q^{-1})\;,
$$
and coproduct/counit/antipode defined by
$$
\begin{array}{c}
\Delta K=K\otimes K\;,\quad \Delta E=E\otimes K+K^{-1}\otimes E\;, 
\\ \rule{0pt}{3.5ex}\epsilon(K)=1\;,\quad \epsilon(E)=0\;,\\ \rule{0pt}{3.5ex}
S(K)=K^{-1}\;,\quad S(E)=-qE\;.
\end{array}
$$
One passes to the notations of \cite{DLS05} with the change $e=-F$, $f=-E$, $k=K$.
The right canonical action is given, on generators $\alpha,\beta$ of $\OO(SU_q(2))$, by
\begin{align*}
\alpha\za K &=q^{\frac{1}{2}}\alpha\;, &
\alpha\za E &=-q\beta^* \;,&
\alpha\za F &=0 \;,\\
\beta\za K &=q^{\frac{1}{2}}\beta\;, &
\beta\za E &=\alpha^* \;,&
\beta\za F &=0 \;.
\end{align*}
We finally need recalling that the representation theory of $\Uq{su}{2}$ is well known 
(cf.~\cite[Thm.~13]{KS97}). In particular, we are interested in
the irreducible representations in which $K$ has positive spectrum: these 
are labelled by an \emph{integer} $n\in\N$ with the representation space 
$V_n$ of dimension $n+1$. In each of these the Casimir element of $\Uq{su}{2}$, 
\begin{equation}\label{Cas}
\mathcal{C}_q=\left(\frac{\smash[t]{q^{\frac{1}{2}}}K-\smash[t]{q^{-\frac{1}{2}}K^{-1}}}{q-q^{-1}}\right)^2+FE \; ,
\end{equation}
has value $\mathcal{C}_q\big|_{V_n}=[\tfrac{n+1}{2}]^2\, \id_{V_n}$.

%%% ======================================================================

\section{K-theory and K-homology}\label{sec:3}

\subsection{Equivariant modules and representations}\label{sec:3.0}
Similarly to the construction of equivariant vector bundles  
associated to a principal bundle on a manifold, here we construct modules --- that we interpret as 
sections of virtual `noncommutative equivariant vector bundles' --- as follows.
Let $\sigma:\Uq{u}{n}\to\mathrm{End}(\C^k)$ be a $*$-representation.
The analogue of (sections of) the vector bundle associated to $\sigma$ is the
$\OO(\CP^n_q)$-module 
$\E (\sigma)$ of elements of $\OO(SU_q(n+1))\otimes\C^k$ which are $\Uq{u}{n}$-invariant for the
Hopf tensor product of the actions $\mL$ and $\sigma$. That is, $\psi\in
\OO(SU_q(n+1))\otimes \C^k$ belongs to $\E(\sigma)$ if and only if
\begin{equation}\label{eq:Esigma}
(\mL_{x_{(1)}} \otimes \sigma(x_{(2)}))\psi=\epsilon(x) \psi \;,
\qquad\forall\;x\in\Uq{u}{n} \;.
\end{equation}
As this set is stable under (left and right) multiplication by an
$\Uq{u}{n}$-
element of $\OO(SU_q(n+1))$, one has that $\E(\sigma)$ is an
$\OO(\CP^n_q)$-bimodule. It is a left $\OO(\CP^n_q)\rtimes\Uq{su}{n+1}$-module as well,
due to the `$\az$' and $\mL$ actions commuting.

Of particular importance are `line bundles', --- bimodules $\E(\sigma)$
coming from one-dimensional representations of $\Uq{u}{n} \simeq\Uq{su}{n}\oplus U(\mathfrak{u}(1))$,
non-trivial only on the $\mathfrak{u}(1)$.
Since the fixed point algebra $\OO(SU_q(n+1))^{\Uq{su}{n}}$ coincides with
$\OO(S^{2n+1}_q)$, (section of) noncommutative
line bundles can be equivalently described as associated to the noncommutative
$U(1)$-principal bundle $S^{2n+1}_q\to\CP^n_q$ via an irreducible
representation of $U(1)$. These are labelled by $N\in\Z$, and the general
line bundle, that we denote by $\Gamma_N$, is given in \S4 of \cite{DD09}.
They are all finitely generated and projective (as one-sided modules), as we shall
explain in detail in \S\ref{sec:3.2}. Note that $\Gamma_0=\OO(\CP^n_q)$. 

The expressions are particularly simple for $n=1$. In this case
\begin{equation}\label{eq:GammaN}
\Gamma_N=\big\{ a\in\OO(SU_q(2))\,\big|\,\mL_K(a)=q^{\frac{N}{2}}a\big\} \;.
\end{equation}
As a left $\Uq{su}{2}$-module, we have a decomposition
(cf.~\S2.2 of \cite{Wag09}, where $\Gamma_N$ is denoted $M_{-N}$):
\begin{equation}\label{deco}
\Gamma_N\simeq\bigoplus_{n-|N|\in 2\N}V_n \;,
\end{equation}
where $V_n$ is the irreducible representation of dimension $n+1$ mentioned before.

An $\OO(\CP^n_q)$-valued Hermitian structure on $\E(\sigma)$ is obtained
by restriction of the canonical Hermitian structure of
$\OO(SU_q(n+1))\otimes \C^k$, that is
$$
(\psi,\eta)_{\E(\sigma)}:=\sum\nolimits_{i=1}^k\psi_i^*\eta_i \;,
$$
for all $\psi=(\psi_1,\ldots,\psi_k)$ and
$\eta=(\eta_1,\ldots,\eta_k)$, with $\psi_i,\eta_i\in\OO(SU_q(n+1))$.

If instead of `representative functions' $\OO(SU_q(n+1))$ one works
with the associated universal $C^*$-algebra $C(SU_q(n+1))$  
of `continuous functions', the above construction yields a full right
Hilbert module over the $C^*$-algebra $C(\CP^n_q)$. Note that
left multiplication by $a\in C(\CP^n_q)$ satisfy
\begin{equation}\label{eq:two}
(a\psi,a\psi)_{\E(\sigma)}=\sum\nolimits_{i=1}^k\psi_i^*a^*a\psi_i 
\leq \|a\|^2\sum\nolimits_{i=1}^k\psi_i^*\psi_i =\|a\|^2(\psi,\psi)_{\E(\sigma)}
\;,
\end{equation}
since $a^*a\leq\|a\|^2$ and conjugation with elements of a $C^*$-algebra
preserves the positivity of an operator. Thus $\E(\sigma)$ is a Morita
equivalence bimodule between $C(\CP^n_q)$ and $\mathrm{End}_{C(\CP^n_q)}\E(\sigma)$
(cf.~\cite[App.~A.3 and A.4]{Lan02}).

By composing the Hermitian structure with the Haar state one gets a pre-Hilbert space
with inner product
\begin{equation}\label{eq:inform}
\inner{\psi,\psi'}:=h\circ(\psi,\psi')_{\E(\sigma)} \;.
\end{equation}
From \eqref{eq:two}, it follows 
$\inner{a\psi,a\psi'}\leq\|a\|^2\inner{\psi,\psi}$, so that one has
a bounded representation of $C(\CP^n_q)$ on the Hilbert space
completion of each of these equivariant modules.
These are the representations used in \S\ref{sec:equivST} for the construction of
covariant differential calculi and equivariant spectral triples 
on $\CP^n_q$.

\subsection{K-theory}\label{sec:3.2}
At the $C^*$-algebra level, by viewing $C(\CP^n_q)$ as the Cuntz--Krieger algebra of a graph \cite{HS02} one proves that $K_0(C(\CP^n_q))\simeq\Z^{n+1}$ (and
$K_1(C(\CP^n_q)) = 0$). The group $K_0$ is given as the cokernel of the incidence matrix canonically associated with the graph. The dual result for K-homology is obtained in an analogous way with the group $K^0$ being the kernel of the transposed matrix \cite{Cun84}; this leads to $K^0(C(\CP^n_q))=\Z^{n+1}$ (and $K^1(C(\CP^n_q))=0$).

Somewhat implicitly, in \cite{HS02} there appear generators of the $K_0$ groups of $C(\CP^n_q)$ as projections in $C(\CP^n_q)$ itself. In \cite{DL09a}
we gave generators of $K_0(C(\CP^n_q))$ in the form of `polynomial functions', so they represent elements of $K_0(\OO(\CP^n_q))$ as well. These latter generators are constructed as follows.

With $N\in\Z$, denote by $\Psi_N=(\psi^N_{j_0,\ldots,j_n})$ the vector-valued `function' on $S^{2n+1}_q$ with $\binom{|N|+n}{n}$ components given by:
\begin{subequations}
\begin{align}
\psi^N_{j_0,\ldots,j_n} &:=[j_0,\ldots,j_n]!^{\frac{1}{2}}q^{-\frac{1}{2}\sum_{r<s}j_rj_s}
(z_0^*)^{j_0}\ldots (z_n^*)^{j_n} \;, &\mathrm{if}\;N\geq 0\;, \\[4pt]
\psi^N_{j_0,\ldots,j_n} &:=[j_0,\ldots,j_n]!^{\frac{1}{2}}q^{\frac{1}{2}\sum_{r<s}j_rj_s+\sum_{r=0}^nrj_r}
z_0^{j_0}\ldots z_n^{j_n} \;,&\mathrm{if}\;N<0\;,\label{psiN}
\end{align}
\end{subequations}
and labeled by non-negative integers satisfying $j_0+\ldots+j_n=|N|$. 
Then $\Psi_N^\dag\Psi_N=1$ and 
\begin{equation}\label{pn}
P_N:=\Psi_N\Psi_N^\dag
\end{equation} 
is a projection: $(P_N)^2=P_N=(P_N)^\dag\,$; the proof is in \cite{DD09,DL09a}, and is a generalization of the case $n=2$ in \cite{DL09b}.
In particular $P_1=P$ is the `defining' projection of the algebra $\OO(\CP^n_q)$ of \S\ref{sec:2.2}. 
As we will see in \S\ref{sec:3.2}, the group $K_0$ is generated by the classes of $P_0,P_{-1},\ldots,P_{-n}$.

The components of $\Psi_N$ are a generating family for $\Gamma_N$ as
a left module, as shown in \S4.1 of \cite{DD09}; hence $\Gamma_N$ is
finitely generated and projective as a left module, and the corresponding
projection is $P_N$. Also it is not difficult to prove that
a generating family for $\Gamma_N$ as a right module, is given by the components of $\Psi^\dag_{-N}$;
hence $\Gamma_N$ is finitely generated and projective as a right module too, with corresponding projection $P_{-N}$. For $n=2$, this is Prop.~3.3 of \cite{DL09b} 
(what we call here $\Gamma_N$, following \cite{DD09}, is denoted $\Sigma_{0,-N}$ in \cite{DL09b}).

The projections $P_N$ are `equivariant' in the following sense.
For an homogeneous space, the equivariant $K^0$-group can be defined as the Grothendieck group of the abelian monoid whose elements are equivalence classes of equivariant vector bundles. It has an algebraic version, denoted $K_0^{\U}(\A)$ where $\U$ is a Hopf \mbox{$*$-algebra} and $\A$ a $\U$-module \mbox{$*$-algebra},  
valid in the non-commutative case as well.
Equivariant vector bundles are replaced by one sided (left, say)
$\A\rtimes\U$-modules which are finitely generated and projective as (left) $\A$-modules; these will be simply called ``equivariant projective modules''. Any such a module is given by a pair $(p,\sigma)$, where $p$ is a $k\times k$ idempotent with entries in $\A$, and 
$\sigma:\U\to\mathrm{Mat}_k(\C)$ is a representation with the following compatibility requirement 
satisfied (see e.g.~\cite[Sect.~2]{DDL08}):
\begin{equation}\label{eq:cov}
(x_{(1)}\az p)  \sigma(x_{(2)})^t=\sigma(x)^t p \;, \qquad \textup{for all} \quad x\in\U\;,
\end{equation}
with `$\phantom{|}^t$' denoting transposition.
The corresponding module $\E=\A^kp$ is made of row vectors elements $v=(v_1,\ldots,v_k)\in\A^k$ 
in the range of the idempotent, $vp=v$, with module structures
$$
(a.v)_i:=av_i\;,\qquad
(x.v)_i:=\sum\nolimits_{j=1}^k(x_{(1)}\az v_j)\sigma_{ij}(x_{(2)})\;, \quad i=1, \dots k \;,
$$
for all $a\in\A$ and $x\in\U$.
An equivalence between two equivariant modules is simply an invertible left $\A\rtimes\U$-module map between them. The group $K_0^{\U}(\A)$ is defined as the Grothendieck group of the abelian monoid whose elements are equivalence classes of $\U$-equivariant projective $\A$-modules; the monoid operation is the direct sum, as usual.

There is an isomorphism $\Gamma_N\simeq\OO(\CP^n_q)^{k_{N,n}}P_N$, with $k_{N,n}={\binom{|N|+n}{n}}$, 
so that $P_N$ are candidates to represent elements in equivariant K-theory.
In fact, it is more convenient to take idempotents $P_N'=R_NP_N R_N^{-1}$, that
are conjugated to $P_N$ through the diagonal matrix $R_N$ having component
$$
q^{\frac{1}{2}\sum_{i=1}^ni(n+1-i)(j_{i-1}-j_i)}
=q^{\frac{1}{2}\sum_{i=0}^n(n-2i)j_i}
$$
in position $(j_0,\ldots,j_N)$.
The need to use idempotents that are not self-adjoint is explained
in Lemma 2.7 of \cite{DDL08}: the module map
$\Gamma_N\to\OO(\CP^n_q)^{k_{N,n}} P_N$ is not unitary,
while the map $\Gamma_N\to\OO(\CP^n_q)^{k_{N,n}}P_N'$ is.
For $n=1$, these are exactly the projections in \cite[Eq.~(33)]{Wag09}.

\begin{prop}\label{prop:psigma}
The pair $(P'_N,\sigma^N)$ is the representative of an element in $K_0^{\Uq{su}{n+1}}(\OO(\CP^n_q))$. Here,
$\sigma^N$ is the irreducible representation %of $\Uq{su}{n+1}$
with highest weight $(N,0,\ldots,0)$ if $N\geq 0$, or
with highest weight $(0,\ldots,0,-N)$ if $N<0$.
\end{prop}
\begin{proof}
Let us give the proof for $N<0$, the one for $N\geq 0$ being similar.
We write the components of $\Psi_N$ as $\psi^N_J$, where
$J=(j_0,\ldots,j_n)$ is a multi-index.

The explicit formul{\ae} for the action on $z_j$'s
are in \S4 of \cite{DD09}. One has
\begin{equation}\label{lact}
E_i\az z'_j=\delta_{i,j}z'_{j+1} \;,\qquad
F_i\az z'_j=\delta_{i,j+1}z'_i \;,\qquad
K_i\az z'_j=q^{\frac{1}{2}(\delta_{i+1,j}-\delta_{i,j})}z'_j \;,
\end{equation}
where $\{E_i,F_i,K_i\}$ are the generators of $\Uq{su}{n+1}$ and
$z'_j$ of $\OO(S^{2n+1}_q)$ in the notations of \cite{DD09}, with 
$i=1,\ldots,n$, and $j=1,\ldots,n+1$; 
we recall that our present notations differ from the ones in \cite{DD09}
for a replacement $z_i=z'_{n+1-i}$.

For a fixed $N$, let $V_N$ be the linear span of the components $\psi^N_J$ of $\Psi_N$. 
Then $V_N$ carry a representation of $\Uq{su}{n+1}$.
And $\psi^N_{-N,0,\ldots,0}=(z_0)^{-N}$ is the highest
weight vector of the representation $(0,\ldots,0,-N)$,
as $K_i\az(z_0)^{-N}=0$ for all $i\neq n$ and $K_n\az(z_0)^{-N}=q^{-\frac{1}{2}N}(z_0)^{-N}$, 
and $E_i\az (z_0)^{-N}=0$ for all $i$.
Hence the representation $\widetilde{\sigma}^N$ on $V_N$ defined by
\begin{equation}\label{eq:sigmaN}
x\az \psi^N_J=\sum\nolimits_{J'}\psi^N_{J'}\, \widetilde{\sigma}^N_{J',J}(x)
\end{equation}
contains the irreducible representation $(0,\ldots,0,-N)$.
Having the latter dimension $k_{N,n}=\binom{|N|+n}{n}$ by Weyl character
formula (cf.~Lemma 3.4 of \cite{DD09}), and being this 
the dimension of $V_N$, the two representations coincide.
Let
$$
\sigma^N(x):=R_N\, \widetilde{\sigma}^N(x)\, R_N^{-1} \;.
$$
In matrix notations, \eqref{eq:sigmaN} becomes
$x\az R_N\Psi_N=\sigma^N(x)^t\hspace{1pt}R_N\Psi_N$,
thinking of
$\Psi_N$ as a column vector and with row-by-column
multiplication understood. Also,
\begin{align*}
(x\az \Psi_N^\dag) \, R_N^{-1} &=
\big(S(x)^*\az \Psi_N\big)^\dag R_N^{-1}
=\big(\widetilde{\sigma}^N(S(x)^*)^t\Psi_N\big)^\dag R_N^{-1} \\
&=\Psi_N^\dag\widetilde{\sigma}^N(S(x))^tR_N^{-1}
=\Psi_N^\dag R_N\sigma^N(S(x))^t R_N^{-2}
 \;.
\end{align*}
Therefore
\begin{align*}
(x_{(1)}\az P_N')\sigma^N(x_{(2)})^t &=
(x_{(1)}\az R_N\Psi_N)(x_{(2)}\az \Psi_N^\dag R_N^{-1})\sigma^N(x_{(3)})^t \\
&=\sigma^N(x_{(1)})^t P_N' R_N^2\sigma^N(S(x_{(2)}))^tR_N^{-2}
\sigma^N(x_{(3)})^t \;.
\end{align*}
We need (also for later use in \S\ref{sec:exist}), the element $K_{2\rho}$
--- implementing the square of the antipode --- and given  in \cite[eq.~(3.2)]{DD09}:
\begin{equation}\label{eia}
K_{2\rho} = \left(K_1^n K_2^{2(n-1)} \ldots K_j^{j(n-j+1)} \ldots K_n^n  \right)^2 \;.
\end{equation}
For now, one readily checks that
$K_{2\rho}\az \Psi_N=
q^{\sum_{i=1}^ni(n+1-i)(j_{i-1}-j_i)}\Psi_N$,
so that $R_N=\sigma^N(K_{2\rho}^{\frac{1}{2}})^t$ and 
\begin{equation}\label{eq:RN2}
R_N^2\, \sigma^N(S(x))^t\, R_N^{-2}=\sigma^N(K_{2\rho}^{-1}S(x)K_{2\rho})
=\sigma^N(S^{-1}(x))^t
\end{equation}
for all $x\in\Uq{su}{n+1}$, by (3.3) of \cite{DD09}.
Thus,
\begin{align*}
(x_{(1)}\az P_N')\, \sigma^N(x_{(2)})^t
&=\sigma^N(x_{(1)})^t P_N' \sigma^N(S^{-1}(x_{(2)}))^t
\sigma^N(x_{(3)})^t \\
&=\sigma^N(x_{(1)})^t P_N' \sigma^N(x_{(3)}S^{-1}(x_{(2)}))^t \\
&=\sigma^N(x_{(1)})^t P_N'\epsilon(x_{(2)})=\sigma^N(x)^t P_N' \;,
\end{align*}
that is exactly \eqref{eq:cov}.
\cvd
\end{proof}

%%% ======================================================================

\subsection{Fredholm modules and Chern characters}\label{sec:3.1}
As we already mentioned, for $\CP^n_q$ the group $K^1$ is trivial.
Here we describe the group $K^0$, whose elements are represented by
Fredholm modules over $\OO(\CP^n_q)$. An even Fredholm
module $(\A,\HH,F,\gamma)$ over a \mbox{$*$-algebra} $\A$ is a $\Z_2$-graded Hilbert
space $\HH=\HH_+\oplus\HH_-$ (with grading operator $\gamma$), 
together with a graded representation $\pi=\pi_+\oplus\pi_-$
of $\A$ on $\HH$ and an odd bounded self-adjoint operator $F$ such that $F^2=1$
and $[F,\pi(a)]$ is a compact operator for all $a\in\A$.
If $[F,\pi(a)]$ is of trace class for all $a\in\A$, we say that the Fredholm
module is $1$-summable.  
The representation symbol will be usually omitted.

Among the generators of $K^0(\OO(\CP^n_q))$ there is one 
whose representation is faithful, which we call `top' Fredholm module
and describe firstly.

Let $\underline{m}=(m_1,\ldots,m_n)\in\N^n$ and
$\ket{\underline{m}}$ be the canonical orthonormal basis of $\ell^2(\N^n)$.
For $0\leq i<k\leq n$, we denote by $\underline{\varepsilon}^k_i\in\{0,1\}^n$
the array\vspace{-5pt}
$$
\underline{\varepsilon}_i^k:=
(\,\,\stackrel{i\;\mathrm{times}}{\overbrace{0,0,\ldots,0}}\,,
\stackrel{k-i\;\mathrm{times}}{\overbrace{1,1,\ldots,1}}\,,
\stackrel{n-k\;\mathrm{times}}{\overbrace{0,0,\ldots,0}}) \;.
$$

\begin{df}[\cite{DL09a}]
Let $0\leq k\leq n$ and $\mc{V}^n_k\subset\ell^2(\N^n)$ be the 
linear span of basis vectors $\ket{\underline{m}}$ satisfying the constraints
$0\leq m_1\leq m_2\leq\ldots\leq m_k$ and $m_{k+1}>m_{k+2}>\ldots>m_n\geq 0$,
with $m_0:=0$. For any $k>0$,
a representation $\pi_{n,k}:\OO(S^{2n+1}_q)\to\B(\ell^2(\N^n))$
is defined as follows.
We set $\pi_{n,k}(z_i)=0$ for all $i>k\geq 1$, while the remaining generators are
\begin{align*}
\pi_{n,k}(z_i)\ket{\underline{m}} &=q^{m_i}\sqrt{1-q^{2(m_{i+1}-m_i+1)}}
   \ket{\underline{m}+\smash[t]{\underline{\varepsilon}_i^k}}\;, && \textup{for} \quad 0\leq i\leq k-1 \;,\\
\pi_{n,k}(z_k)\ket{\underline{m}} &=q^{m_k}\ket{\underline{m}} \;,
\end{align*}
on the subspace $\mc{V}^n_k\subset\ell^2(\N^n)$, and they are zero on the orthogonal subspace. When $k=0$, we define $\pi_{n,0}(z_i)=0$ if $i>0$, 
while 
\begin{align*}
\pi_{n,0}(z_0)\ket{\underline{m}} &=\ket{\underline{m}} \;, && \textup{for} \quad m_1>m_2>\ldots>m_n\geq 0 , \\
\pi_{n,0}(z_0)\ket{\underline{m}}& =0 \;, && \textup{otherwise}.
\end{align*}
\end{df}
Each representaton $\pi_{n,k}$ is an irreducible $*$-representation of both
$\OO(S^{2n+1}_q)$ and $\OO(\CP^n_q)$ when restricted to $\mc{V}^n_k$, and is identically zero outside $\mc{V}^n_k$. Most importantly, if $|j-k|>1$ we have
\cite{DL09a}:
$$
\pi_{n,j}(a) \,\pi_{n,k}(b)=0 \;,\qquad\forall\;
a,b\in\OO(S^{2n+1}_q)\;.
$$ 
As a consequence, the maps $\pi_\pm:\OO(S^{2n+1}_q)\to\B(\ell^2(\N^n)))$, defined by
$$
\pi_+^{(n)}(a):=\sum_{\substack{ k\;\mathrm{even} \\ 0\leq k\leq n }}\pi_{n,k}(a) \;,\qquad
\pi_-^{(n)}(a):=\sum_{\substack{ k\;\mathrm{odd} \\ 0\leq k\leq n }}\pi_{n,k}(a) \;,
$$
are representations of the algebra $\OO(S^{2n+1}_q)$ and, by restriction,
of $\OO(\CP^n_q)$.

An even Fredholm module for $\OO(\CP^n_q)$ is obtained
with the representation $\pi_n:=\pi_+^{(n)}\oplus\pi_-^{(n)}$ on $\HH_n := 
\ell^2(\N^n)\oplus\ell^2(\N^n)$, obvious grading operator $\gamma_n$,
and
$$
F_n:=\ma{0\; & 1 \\ 1\; & 0} \;.
$$
Its $1$-summability follows from the proposition below \cite{DL09a}.

\begin{prop}\label{prop:3.3}
The difference $\pi_+^{(n)}(a)-\pi_-^{(n)}(a)$ is of trace class on $\HH_n$ for all $a\in\OO(\CP^n_q)$. Furthermore, the trace is given by a series which --- as a function of $q$ --- is absolutely convergent in the open interval $0<q<1$.
\end{prop}

Additional $n$ Fredholm modules $(\OO(\CP^n_q),\HH_k,F_k,\gamma_k)$, $0\leq k< n$,
are obtained using the \mbox{$*$-algebra} morphism $\OO(\CP^n_q)\to\OO(\CP^k_q)$, restriction of the morphism $\OO(S^{2n+1}_q)\to\OO(S^{2k+1}_q)$ given by the map sending
to zero the generators $z_{k+1},z_{k+2},\ldots,z_n$.
With this map, one pull-backs the `top' Fredholm module of $\CP^k_q$
to $\CP^n_q$. For $k=0$, we set $\OO(\CP^0_q):=\C$ and
the `top' Fredholm module --- the generator of $K^0(\C)$ ---, is given by the
(non-unital) representation $\C \ni a\mapsto a\oplus 0$ on $\HH_0:=\mathbb{C}\oplus\mathbb{C}$,
with grading $\gamma_0=1\oplus -1$ and $F_0$ the operator interchanging the
two components, $F_0(x\oplus y)=y\oplus x$ for all $x,y\in\mathbb{C}$.

The pairing of the K-homology class $[F_k]$ of $(\OO(\CP^n_q),\HH_k,F_k,\gamma_k)$
with an element $[p]\in K_0(\OO(\CP^n_q))$ is given by:
$$
\inner{[F_k], [p]}=\tfrac{1}{2}\,\mathrm{Tr}(\gamma_kF_k[F_k,p]) \;.
$$
In particular, for $k=0$ this computes the dimension of the
fiber of the restriction of noncommutative vector bundles over $\CP_q^n$ to the
`classical point' of $\CP_q^n$ (given by the unique character of the algebra):
the computation yields the `rank' of the corresponding projective module.

\begin{prop}\label{q-ind}
For any $N\in\N$ and for all $0\leq k\leq n$,
the pairing between the K-theory classes $[P_{-N}]$ of the (line bundle)
projections $P_{-N}$ described in \S\ref{sec:3.2} and the K-homology classes $[F_k]$ is:
$$
\inner{[F_k],[P_{-N}]}=\tbinom{N}{k} \;,
$$
with $\binom{N}{k}:=0$ when $k>N$.
\end{prop}
For the proof in \cite{DL09a} one computes the pairing by evaluating the series giving the trace in
the $q\to 0$ limit, being the series absolutely convergent as in Prop.~\ref{prop:3.3}. 
For $q\to 0$ only finitely many terms survive, and the final result easily follows.
With the above result, we proved in \cite{DL09a} that the elements $[F_0],\ldots,[F_n]$ are generators of $K^0(\OO(\CP^n_q))$, and the elements $[P_0],\ldots,[P_{-n}]$ are generators of $K_0(\OO(\CP^n_q))$.
In particular, similarly to the classical situation, the K-theory is generated by
line bundles.

%%% ======================================================================

\section{Dirac operators and spectral triples}\label{sec:equivST}

\subsection{Regular spectral triples}
Spectral triples, or ``unbounded Fredholm modules'', provide a non-commuta\-ti\-ve generalization of the notion of closed Riemannian orientable (or spin$^c$) manifold \cite{Con94,Con08}.
A unital spectral triple is the datum $(\A,\HH,D)$ of a \mbox{$*$-algebra} $\A$ with a bounded representation $\pi:\A\to\B(\HH)$ on a Hilbert space $\HH$, and a selfadjoint operator $D$ on $\HH$ --- the `Dirac' operator --- with compact resolvent, such that $[D,\pi(a)]$ is bounded for all $a\in\A$.
The spectral triple is called \emph{even} if $\HH=\HH_+\oplus\HH_-$ is $\Z_2$-graded and,
for this decomposition, $\pi(\A)$ is diagonal while the operator $D$ is off-diagonal. 
We denote by $\gamma$ the grading operator, and set $\gamma=1$ when the spectral triple is \emph{odd} (no grading then). 
The compact resolvent requirement for the Dirac operator guarantees, for example, in the even case that 
the twisting of $D_\pm=D|_{\HH_\pm}$ with projections are Fredholm operators: a crucial property for the
construction of `topological invariants' via index computations \cite{Con94}.
If there is a $d\in\R^+$ such that $(1+D^2)^{-d/2}$ is in the Dixmier ideal $\mL^{(1,\infty)}(\HH)$,
the spectral triple is said to have ``metric dimension'' $d$ or to be $d$-summable (cf.~Chap.~4 of \cite{Con94}). 

While spectral triples correspond to spin$^c$ or orientable Riemannian manifolds,
\emph{real} spectral triples correspond to manifolds that are spin \cite{Con95}.
A spectral triple $(\A,\HH,D,\gamma)$ is \emph{real} if there is in addition an antilinear isometry $J:\HH\to\HH$, called
the \emph{real structure}, such that
\begin{equation}\label{eq:yyy}
J^2=\epsilon 1\;,\qquad JD=\epsilon' DJ\;,\qquad
J\gamma=\epsilon''\gamma J\;,
\end{equation}
and
\begin{equation}\label{eq:real}
[a,JbJ^{-1}]=0\;,\qquad
[[D,a],JbJ^{-1}]=0\;,
\end{equation}
for all $a,b\in\A$.
The signs $\epsilon$, $\epsilon'$ and $\epsilon''$ determine the KO-dimension (an integer modulo $8$)
of the triple \cite{Con95}.
In some examples (not in the present case) conditions \eqref{eq:real} have to be slightly relaxed 
(see for instance \cite{DLS05}).

As conformal structures are classes of (pseudo-)Rie\-mannian metrics,
similarly Fredholm modules are ``conformal classes'' of spectral triples
\cite{Con94,Bar07}: given
a spectral triple $(\A,\HH,D)$, a Fredholm module $(\A,\HH,F)$ can be obtained
by replacing $D$ with the bounded operator $F:=D(1+D^2)^{-\frac{1}{2}}$ (one can use $F:=D|D|^{-1}$ if $D$ is invertible), and viceversa any K-homology class has a representative that arises from a spectral triple through this construction \cite{BJ83}.
Passing from bounded to unbounded Fredholm modules is convenient since it allows to use powerful tools such 
as local index formul{\ae} \cite{CM95}. 

From now on we shall only consider spectral triples whose
representation is faithful, identify $\A$ with $\pi(\A)$ and omit  the representation symbol.
For the space $\CP_q^n$, we introduced in \cite{DL09a} even spectral
triples of any metric dimension $d\in\mathbb{R}^+$ whose conformal class is 
the top Fredholm module $(\OO(\CP^n_q),\HH_n,F_n,\gamma_n)$ in \S\ref{sec:3.1}. 
They are constructed by giving explicitly the spectral decomposition of the Dirac operator.
For example, an $n$-dimensional spectral triple is obtained by taking $D=|D|F_n$ on $\HH_n$ with 
 $$
|D|\ket{\underline{m}}:=(m_1+\ldots+m_n)\ket{\underline{m}} \;.
$$
The eigenvalues are $\pm\lambda$, for $\lambda\in\N$, with multiplicity $\binom{\lambda+n}{n-1}$. This is a polynomial in $\lambda$ of order $n-1$, and so the metric dimension is $n$, as claimed.

\subsection{Equivariant spectral triples}
When $\A$ is a $\U$-module \mbox{$*$-algebra}, for some Hopf \mbox{$*$-algebra} $\U$, one may consider
spectral triples with ``symmetries'', describing the analogue of homogeneous spin structures.
A unital spectral triple $(\A,\HH,D,\gamma)$ is called $\U$-equivariant if\\ 
(i) there is a dense subspace $\mc{M}\subset\mathrm{Dom}(D)$ of $\HH$ where the representation of $\A$ can be extended to a representation of $\A\rtimes\U$, \\
(ii) both $D$ and $\gamma$ commute
with $\U$ on $\mc{M}$. \\
In case there is a real structure $J$, one further asks that
$J|_{\mc{M}}$ is the antiunitary part of a (possibly unbounded) 
antilinear operator $\widetilde{J}:\mc{M}\to\mc{M}$ such that
\begin{equation}\label{inp-ant}
\widetilde{J}\,x=S(x)^*\widetilde{J} \;,\qquad\forall\;x\in\U\;.
\end{equation}
In other words, the antilinear involutive automorphism
$x \mapsto S(x)^*$ of the Hopf \mbox{$*$-algebra}~$\U$ is implemented by the 
operator $\widetilde{J}$. This resonate with a known feature of quantum-group duality 
in the $C^*$-algebra setting of \cite{MNW03}, where, in that setting, it is discusses the relation of the Tomita operator with the antipode $S$ and the $^*$ structure of a quantum group.

As already mentioned, spectral triples for $\CP^n_q$ were constructed
in \cite{DL09a}. They  
are typical of the noncommutative case, as they have no $q\to 1$ analogue.
Although they are not equivariant, they are ``regular'' in the sense of \cite{CM95}.

On the other hand, even regular spectral triples on \mbox{$q$-spaces} usually don't give very interesting local index formul{\ae};
tipically the unique term surviving in Connes-Moscovici local cocycle \cite{CM95} is the non-local one (cf.~\cite{DDLW07,DDL08}).
On $\CP^1_q$ a more complicated local index formula is given in \cite{NT05},
and is obtained using the \emph{non-regular} and \emph{equivariant} spectral triple of \cite{DS03}.
The geometrical nature of the latter is explained in \cite{SW04} where it is also implicitly suggested how to generalize the construction to $\CP^n_q$, by using the action of the Hopf algebra $\Uq{su}{n+1}$,
in particular the action of quasi-primitive elements, which are external derivations on $\CP^n_q$. 
This idea was used in \cite{Kra04} to
construct --- on any quantum irreducible flag manifold, including then $\CP^n_q$
--- a Dirac operator $D$ that realizes by commutators 
the unique covariant (irreducible, finite-dimensional)
first order \mbox{$*$-calculus} of \cite{HK04} for $\CP^n_q$. In particular, the exterior derivative 
is given by $\delta(a)=\sqrt{-1}\, [D,a]$, for $a\in\CP^n_q$ (the coefficient $\sqrt{-1}$
is inserted to get a real derivation). It is not clear whether this leads to
a spectral triples or not, as the compact resolvent condition is yet unproven.

Equivariant spectral triples on $\CP_q^n$ are constructed in \cite{DD09}, in
complete analogy with the $q=1$ case, by using the fact that complex projective
spaces are K{\"a}hler manifolds: in particular they admit a homogeneous (for the action of $SU(n+1)$) K{\"a}hler
metric, the \emph{Fubini-Study metric}.
The result is a family of (equivariant, even) spectral triples $(\OO(\CP^n_q),\HH_N,D_N,\gamma_N)$
labelled by $N\in\Z$ (Although we use the same symbol, the Hilbert
spaces here are not the Hilbert spaces $\HH_k$ of \S\ref{sec:3}.).  The space $\HH_0$ are the noncommutative analogue of
$(0,1)$-forms (in fact, they give a finite-dimensional covariant
differential calculus on $\CP^n_q$) with $D_0$ the analogue of the Dolbeault-Dirac
operator; $\HH_N$ is the tensor product of $\HH_0$ with `sections of
line bundles' with monopole charge $N$ over $\CP^n_q$ with $D_N$ 
the twisting of $D_0$ by the Grassmannian connection of the line bundle.
If $n$ is odd, for $N=\frac{1}{2}(n+1)$ one has 
a \emph{real} spectral triple whose Dirac operator is a deformation of
the Dirac operator of the Fubini-Study metric, in parallel with $\CP^n$ being a spin manifold when $n$ is odd.

The spectrum of $D_N$ is computed by relating its square $D_N^2$ to the Casimir of
$\Uq{su}{n+1}$. One finds that for $q<1$ the
eigenvalues of $D_N$ grow exponentially, hence the spectral triple is of metric dimension $0^+$, or better 
$0^+$-summable. For $q=1$ one finds (as expected) the spectrum given in \cite{DHMOC08}.

\section{The projective line  $\CP^1_q$  as a noncommutative manifold}\label{plncm}

Recall the $\OO(\CP^1_q)$-bimodules $\Gamma_N$ in \eqref{eq:GammaN}.
As proven in \cite{DS03}, modulo unitary equivalences
there is a unique real equivariant even spectral triple
for $\CP^1_q$ on the Hilbert space completion of $\Gamma_1\oplus\Gamma_{-1}$.
This spectral triple has metric dimension $0$ (there is
no real equivariant spectral triple on $\Gamma_1\oplus\Gamma_{-1}$ with summability
different from $0^+$). Its geometrical nature --- that the Dirac operator
is coming from the right action of $\Uq{su}{2}$ ---,
is explained in \cite{SW04}. %and is recalled in \cite[Sect.~2]{DD09}.

Twisting the Dirac operator on the tensor product
of $\Gamma_1\oplus\Gamma_{-1}$ with a line bundle \cite{Si08},
leads to spectral triples which are in general not real \cite[Sect.~2]{DD09}.
In \cite{DDLW07} we constructed spectral triples of any summability
with a real structure $J$ satisfying a weaker version of the reality
and first order condition in \eqref{eq:real}.

In the next section we shall describe a new family of equivariant real
spectral triples for $\CP^1_q$. They generalize the ones of \cite{DS03} and 
are all inequivalent to each other (in particular, not equivalent to the one of \cite{DS03}).

\subsection{ A family of equivariant real spectral triples for $\CP^1_q$}\label{sec:new}
With the $\OO(\CP^1_q)$-bimodules $\Gamma_N$ in \eqref{eq:GammaN}, for $n\in\frac{1}{2}\Z$, let
$W_n$ be the Hilbert space completion of $\Gamma_{-2n}$ with respect to the inner product coming 
from the Haar state of $SU_q(2)$ as in \eqref{eq:inform}.
For a fixed $j\in\N+\frac{1}{2}$, we call $\HH_j$ the space of vectors
$\vec{a}=(a_{-j},a_{-j+1},...,a_j)^t$
with components $a_n\in W_n$, for $n=-j,-j+1,\ldots,j$. 

The
representation of $\OO(\CP^1_q)$ is the obvious left module structure
of $\HH_j$. The Dirac operator
$D_j$, the grading $\gamma_j$ and the real structure $J_j$ are given by
\begin{align*}
D_j\vec{a} &=(\mL_Ea_{-j+1},\, \mL_Fa_{-j},\mL_Ea_{-j+3},\, \mL_Fa_{-j+2},\ldots,
\, \mL_Ea_j, \, \mL_Fa_{j-1})^t \;,\\
\gamma_j\vec{a} &=(-a_{-j},\, a_{-j+1},\, -a_{-j+2},\, a_{-j+3},\ldots,\, -a_{j-1},\, a_j)^t \;,\\
J_j\vec{a} &=K\az (q^{-j}a_j^*,\, -q^{-j+1}a_{j-1}^*, %\, q^{-j+2}a_{j-2}^*, \,-q^{-j+3}a_{j-3}^*,
\ldots,\, q^{j-1}a_{-j+1}^*,\, -q^ja_{-j}^*)^t \;.
\end{align*}
Note that $\gamma_j|_{W_n}$ is $1$ if $j+n$ is odd and
is $-1$ if $j+n$ is even.
\begin{prop}\label{jest}
The datum $(\OO(\CP^1_q),\HH_j,D_j,\gamma_j,J_j)$ is a real
even $\Uq{su}{2}$-equivariant spectral triple, with KO-dimension $2$ and
metric dimension $0^+$.
\end{prop}
\begin{proof}
The proof is analogous to the one in \cite{SW04} 
(and generalizations in \cite{DDL08b,DD09}).
By definition $a_n\in \Gamma_{-2n}$ satisfies
$\mL_K(a_n)=q^{-n}a_n$, and 
$\mL_K\mL_E= q \mL_E\mL_K$
proves that $\mL_E$ is a densely defined
operator $W_n\to W_{n-1}$. Similarly, $\mL_F$
is a densely defined operator $W_n\to W_{n+1}$.
Hence $D_j$ is a well defined symmetric operator
on $\mc{M}:=\bigoplus_n\Gamma_{2n}$. It can be closed
to a self-adjoint operator on $\HH_j$ (in fact, one can
diagonalize it and give its domain of self-adjointness
explicitly).

The representation of $\OO(\CP^1_q)$ is clearly bounded.
From the coproduct formula of $E$, and the defining property
$\mL_K(a)=a$ of $a\in\OO(\CP^1_q)$, one gets
$$
\mL_E(a \eta)=
(\mL_Ea)(\mL_{K^{-1}}\eta)+
(\mL_Ka)(\mL_E \eta)=q^n(\mL_Ea)\eta+a(\mL_E \eta) \;,
$$
for all $a\in\OO(\CP^1_q)$ and $\eta\in\Gamma_{-2n}$. Thus the commutator 
$[\mL_E,a]=q^n\mL_E(a)$, is the multiplication
operator for an element $\mL_E(a)\in\OO(SU_q(2))$, and
hence a bounded operator $W_n\to W_{n-1}$. A similar formula
holds for $[\mL_F,a]$, proving that $[D_j,a]$ is a
bounded operator, for all $a\in\OO(\CP^1_q)$.

The grading commutes with any $a\in\OO(\CP^1_q)$ and anticommutes
with $D_j$. The square of the Dirac operator is
$$
D_j^2\vec{a}=
\big(\,\mL_E \mL_F a_{-j}\,,\,\mL_F \mL_E a_{-j+1}\,,\,\ldots, \, \mL_E\mL_F a_j, \, \mL_F\mL_E a_{j-1} \,\big)^t \;.
$$
Since $\mL_K$ is proportional to the identity on each $\Gamma_{-2n}$,
modulo a constant matrix, $D_j^2$ is given by the $\mL$ action
of the central element $\mc{C}_q$ in \eqref{Cas}.
For a central element left and right canonical actions coincide,
and with respect to the left action of $\Uq{su}{2}$ we have the
decomposition into irreducible representations
$\Gamma_{-2n}\simeq\bigoplus_{\ell-|n|\in\N}V_{2\ell}$,
as given in \eqref{deco}. Since the eigenvalues
of $\mc{C}_q$ grow exponentially with $\ell$, the operator $\mc{C}_q$
has compact resolvent on each $W_n$;
the operator $D_j$ has compact resolvent too, given that there
are only finitely many $W_n$ in $\HH_j$. 
This proves that $(\OO(\CP^1_q),\HH_j,D_j,\gamma_j)$ is a spectral
triple. In fact, the eigenvalues of $D_j$ growing exponentially as well, the operator 
$(1+D_j^2)^{-\epsilon}$ is of trace class for any $\epsilon>0$. Hence
the metric dimension of the spectral triple is $0^+$ (the spectral triple is $0^+$-summable).

Next the real structure. The operator $J_j$ is an isometry:
\begin{align*}
\inner{J_j\vec{a},J_j\vec{b}} &=
\sum\nolimits_nq^{-2n}h\big((K^{-1}\az a_n)(K\az b_n^*)\big) \\
&=\sum\nolimits_n h\big((K^{-1}\az a_n\za K^{-1})(K\az b_n^*\za K)\big) \\
&=\sum\nolimits_n h\big((K\az b_n^*\za K)(K\az a_n\za K)\big) \\
&=\sum\nolimits_n h\big(K\az (b_n^*a_n)\za K\big) \\
&=\sum\nolimits_n h(b_n^*a_n)=\inner{\vec{b},\vec{a}} \;,
\end{align*}
where we used the bi-invariance and the modular property
of the Haar state $h$, i.e.~$h(ab)=h\left( b\, (K^2\az a\za K^2) \right)$
(cf.~eq.~(3.4) in \cite{DD09}). Clearly $J_j\gamma_j=-\gamma_jJ_j$. And,
since $x\az a^*= \left(S(x)^*\az a \right)^*$ for all $x\in\Uq{su}{2}$ and any $a\in\OO(SU_q(2))$, 
one also easily checks that $J_j^2=-1$. As for the antilinear operator $\widetilde{J}$ as in \eqref{inp-ant},
let $\widetilde{J}_j:\mc{M}\to\mc{M}$ be
the (unbounded) operator
$$
\widetilde{J}_j\vec{a}=(q^{-j}a_j^*,-q^{-j+1}a_{j-1}^*,\ldots,q^{j-1}a_{-j+1}^*,-q^ja_{-j}^*)^t \;.
$$
Note that $J_j\vec{a}=K\az(\widetilde{J}_j\vec{a})$. Since $K\az $ is a positive operator,
$J_j$ is the antiunitary part of $\widetilde{J}_j$. 
Furthermore, from $x\az a^*= \left( S(x)^*\az a \right)^* $ and $S(S(x^*)^*)=x$
it follows $S(x^*)\az\widetilde{J}_j(\vec{a})=\widetilde{J}_j(x\az\vec{a})$
for all $x\in\Uq{su}{2}$, i.e.~the relation \eqref{inp-ant}.

Since %$(a^*\za x)^*=a\za S^{-1}(x)^*$, we have
$\mL_E(a^*)=-q^{-1}(\mL_Fa)^*$ and $\mL_F(a^*)=-q(\mL_Ea)^*$, we have:
\begin{align*}
& D_j\widetilde{J}_j\vec{a} =\big(-q^{-j+1}\mL_E(a_{j-1}^*),q^{-j}\mL_F(a_j^*),\ldots,-q^j\mL_E(a_{-j}^*),q^{j-1}\mL_F(a_{-j+1}^*)\big)^t \\
& \quad =\big(q^{-j}(\mL_Fa_{j-1})^*,-q^{-j+1}(\mL_Ea_j)^*,\ldots,q^{j-1}(\mL_Fa_{-j})^*,-q^j(\mL_Ea_{-j+1})^*\big)^t \\
& \quad =\widetilde{J}_j(\mL_Ea_{-j+1},\mL_Fa_{-j},\mL_Ea_{-j+3},\mL_Fa_{-j+2},\ldots,\mL_Ea_j,\mL_Fa_{j-1})^t \\
& \quad =\widetilde{J}_jD_j\vec{a} \;.
\end{align*}
As left and right actions of $\Uq{su}{2}$ commute,
it follows that $[D_j, J_j]=0$.

The signs in \eqref{eq:yyy} are then $\epsilon=-1$, $\epsilon'=1$ and $\epsilon''=-1$ and correspond to KO-dimension $2$. 
One easily checks that $J_jaJ_j^{-1}$ is the operator of right multiplication by $a^*$, for all $a\in\OO(\CP^1_q)$,
and hence it commutes with $b$ and $[D_j,b]$ for $b\in\OO(\CP^1_q)$. This proves both conditions \eqref{eq:real}.

The left action of $\Uq{su}{2}$ on $\mc{M}$ commutes with $D_j$ (it commutes
with the right action), it clearly commute with the grading, since each
$\Gamma_{-2n}$ is a left $\Uq{su}{2}$-module, and in fact the representation
of $\OO(\CP^1_q)$ extends to a representation of $\OO(\CP^1_q)\rtimes\Uq{su}{2}$.
Thus we have a real equivariant spectral triple, as claimed.\cvd
\end{proof}

Spectral triples of Prop.~\ref{jest} corresponding to
different values of $j$ are `topologically' inequivalent since, as we shall see in next section, they give different
values when paired with the generator $p=P_1$ of $K_0(\OO(\CP^1_q))$,
\begin{equation}\label{eq:P1}
p=\ma{\alpha^*\alpha\; & \alpha^*\beta \\ \beta^*\alpha\; & \beta^*\beta} =\ma{1-q^2A & \;B^* \\ \;B & A} \;,
\end{equation}
having used generators $A=\beta^*\beta$ and $B=\beta^*\alpha$ for the algebra $\OO(\CP^1_q)$. 

In the next section this result will be also used to establish rational Poincar{\'e} duality 
for the spectral triples, thus generalizing the analogous result proven in  \cite{Wag09}  
for the spectral triple of \cite{DS03}.

\subsection{Index computations and rational Poincar\'e duality}\label{sec:7.2}
We consider here the spectral triple $(\OO(\CP^1_q),\HH_j,D_j,\gamma_j,J_j)$ of Prop.~\ref{jest}, 
being $j\in\N+\frac{1}{2}$ a fixed number. With $n\in\Z+\frac{1}{2}$ and condition $|n|\leq j$, let
$$
\HH_j^+:=(1+\gamma_j)\HH_j=\!\bigoplus_{j+n\;\mathrm{odd}}\!W_n \,,\qquad
\HH_j^-:=(1-\gamma_j)\HH_j=\!\bigoplus_{j+n\;\mathrm{even}}\!W_n \,,
$$
and let $D_j^+:=\smash[b]{D_j|_{\HH_j^+}\otimes \id_{\C^2}}$.
Let $p$ be  the `defining' projection in \eqref{eq:P1}.
We aim at computing the index of the (unbounded) operator 
$$
pD^+_jp \,:\,p(\HH_j^+\otimes\C^2)\to p(\HH_j^-\otimes\C^2) \;,
$$
yielding the
pairing of the K-homology class of the spectral triple with
the non-trivial generator of $K_0(\OO(\CP^1_q))$.

\begin{prop}\label{prop:indexj}
It holds that
$$
\mathrm{Index}(pD^+_jp)=
\begin{cases}
\tfrac{1}{2}(j^2-\tfrac{9}{4}) & \mathrm{if}\;j\in 2\N+\frac{1}{2} \;,\\[2pt]
\tfrac{1}{2}(j^2-\tfrac{1}{4}) & \mathrm{if}\;j\in 2\N+\frac{3}{2} \;.
\end{cases}
$$
The index being never zero, these spectral triples are ``topologically" non-trivial.
\end{prop}
\begin{proof}
To compute the index, we look for a ``nice'' basis of $\HH_j^\pm\otimes\C^2$.
We begin by recalling the left regular representation of $\OO(SU_q(2))$,  
found for instance in \cite{DLS05}. An orthonormal basis of $\OO(SU_q(2))$ is given by
$$
\ket{l,m,n}=q^n[2l+1]^{\frac{1}{2}}t^l_{nm}\;,\qquad
l\in\tfrac{1}{2}\N,\;l-|m|\in\N,\;l-|n|\in\N.
$$
with $t^l_{nm}$ the matrix elements of irreducible corepresentations
\cite[Sect.~4.2.4]{KS97} (with respect to the
notations of \cite{DLS05} we exchanged the labels $m$ and $n$).
The left regular representation is given on generators by \cite[Prop.~3.3]{DLS05}
\begin{align*}
\alpha\ket{l,m,n} =      q^{-l+\frac{1}{2}(m+n-1)} \left(\frac{[l+m+1][l+n+1]}{[2l+1][2l+2]}\right)^{\frac{1}{2}} & \ket{l^+,m^+,n^+} \\     +q^{l+\frac{1}{2}(m+n+1)} \left(\frac{[l-m][l-n]}{[2l][2l+1]}\right)^{\frac{1}{2}} & \ket{l^-,m^+,n^+} \;, \\
\beta\ket{l,m,n} =      q^{\frac{1}{2}(m+n-1)} \left(\frac{[l-m+1][l+n+1]}{[2l+1][2l+2]}\right)^{\frac{1}{2}} & \ket{l^+,m^-,n^+} \\
  -q^{\frac{1}{2}(m+n-1)} \left(\frac{[l+m][l-n]}{[2l][2l+1]}\right)^{\frac{1}{2}} & \ket{l^-,m^-,n^+} \;,
\end{align*}
with the notation $k^\pm:=k\pm\tfrac{1}{2}$.
Also, from \cite[eq.~(3.1)]{DLS05} and the definition of the automorphism $\vartheta$ there
(i.e.~$K=k=\vartheta(k^{-1})$, $E=-f=\vartheta(e)$ and $F=-e=\vartheta(f)$) we deduce
\begin{align*}
\mL_K\ket{l,m,n} &=q^{-n}\ket{l,m,n} \;,\\
\mL_F\ket{l,m,n} &=\sqrt{[l-n][l+n+1]}\ket{l,m,n+1} \;,\\
\mL_E\ket{l,m,n} &=\sqrt{[l-n+1][l+n]}\ket{l,m,n-1} \;.
\end{align*}
The Hilbert space $W_n$ has basis $\ket{l,m,n}$, with
$n\in\Z+\frac{1}{2}$ fixed, $l=|n|,|n|+1,\ldots$ and $m=-l,-l+1,\ldots,l$.
Using this, a basis of $W_n\otimes\C^2$ (already employed in \cite[Sect.~3.8]{Dan07}),  
is given by:
\begin{align*}
v^{n,\uparrow}_{l,m} &:=
\frac{1}{\sqrt{[2l]}} \left(\begin{array}{r}
\sqrt{q^{-l+m}[l+m]} \ket{l^-,m^-,n} \\[2pt]
\sqrt{q^{l+m}[l-m]} \ket{l^-,m^+,n} 
\end{array}\right)
 \;, \qquad l=|n|+\tfrac{1}{2},|n|+\tfrac{3}{2},\ldots,
\\[5pt]
v^{n,\downarrow}_{l,m} &:=
\frac{1}{\sqrt{[2l+2]}} \left(\begin{array}{r}
\sqrt{q^{l+m+1}[l-m+1]} \ket{l^+,m^-,n} \\[2pt]
-\sqrt{q^{-l+m-1}[l+m+1]} \ket{l^+,m^+,n}
\end{array}\right) \;,
 \\
& \hspace{7.3cm} l=|n|-\tfrac{1}{2},|n|+\tfrac{1}{2},\ldots,
\end{align*}
where $m=-l,-l+1,\ldots,l$. Notice that in previous equation $l$ and $m$
are integers (while $n$ is not). 
For notational convenience, we set
$v^{n,\uparrow}_{|n|-\frac{1}{2},m}:=0$ and start
counting from $l=|n|-\frac{1}{2}$ for both $v^\uparrow$
and $v^\downarrow$.
An easy exercise checks that  passing from
the vectors $\{\ket{l,m,n}\otimes\binom{1}{0},\ket{l,m,n}\otimes\binom{0}{1}\}$
to the vectors $\{v^{n,\uparrow}_{l,m},v^{n,\downarrow}_{l,m}\}$
is an isometry, and thus we got an orthonormal basis
of $W_n\otimes\C^2$. 

The restriction of the left regular representation of $\OO(SU_q(2))$
to $\OO(\CP^1_q)$ is given on generators $A$ and $B$ by:
\begin{align*}
 A\ket{l,m,n} & = 
-q^{m+n-1}\tfrac{1}{[2l+2]}\sqrt{\tfrac{[l+m+1][l-m+1][l+n+1][l-n+1]}{[2l+1][2l+3]}}  \ket{l+1,m,n} \\
&\phantom{=} +q^{m+n-1}\left(\tfrac{[l-m+1][l+n+1]}{[2l+1][2l+2]}+\tfrac{[l+m][l-n]}{[2l][2l+1]}\right) \ket{l,m,n} \\
&\phantom{=} -q^{m+n-1}\tfrac{1}{[2l]}\sqrt{\tfrac{[l+m][l-m][l+n][l-n]}{[2l-1][2l+1]}} \ket{l-1,m,n} \;,
 \\[10pt]
B\ket{l,m,n} &= 
-q^{-l+m+n-\frac{1}{2}}\tfrac{1}{[2l+2]}\sqrt{\tfrac{[l+m+1][l+m+2][l+n+1][l-n+1]}{[2l+1][2l+3]}}  \ket{l+1,m+1,n} \\
&\phantom{=} +q^{m+n}\tfrac{\sqrt{[l+m+1][l-m]}}{[2l+1]}\left(\tfrac{q^{-l-\frac{1}{2}}[l+n+1]}{[2l+2]} -\tfrac{q^{l+\frac{1}{2}}[l-n]}{[2l]}\right)  \ket{l,m+1,n} \\
&\phantom{=} +q^{l+m+n+\frac{1}{2}}\tfrac{1}{[2l]}\sqrt{\tfrac{[l-m][l-m-1][l+n][l-n]}{[2l-1][2l+1]}} \ket{l-1,m+1,n} \;.
\end{align*}
Using these, for the action of the projection $p$ in \eqref{eq:P1}, we get:
$$
\begin{array}{rrr}
p\,v^{n,\uparrow}_{l,m}=
& q^{-l+n-\frac{1}{2}}\frac{[l+n+\frac{1}{2}]}{[2l+1]} \;v^{n,\uparrow}_{l,m} \;\;+
& q^n\frac{\sqrt{[l+n+\frac{1}{2}][l-n+\frac{1}{2}]}}{[2l+1]} \;v^{n,\downarrow}_{l,m} \;,\\[10pt]
p\,v^{n,\downarrow}_{l,m}=
& q^n\frac{\sqrt{[l+n+\frac{1}{2}][l-n+\frac{1}{2}]}}{[2l+1]} \;v^{n,\uparrow}_{l,m} \;\;+
& q^{l+n+\frac{1}{2}}\frac{[l-n+\frac{1}{2}]}{[2l+1]} \;v^{n,\downarrow}_{l,m} \;,
\end{array}
$$
if $l>|n|-\frac{1}{2}$, while for $l=|n|-\frac{1}{2}$:
$$
p\,v^{n,\downarrow}_{|n|-\frac{1}{2},m}=
\begin{cases}
0 & \mathrm{if}\;n>0\;, \\
v^{n,\downarrow}_{|n|-\frac{1}{2},m} & \mathrm{if}\;n<0\;.
\end{cases}
$$
We rewrite previous equations in the form
\begin{align*}
\begin{pmatrix}
p\, v^{n,\uparrow}_{l,m} \\[2pt] p\, v^{n,\downarrow}_{l,m}
\end{pmatrix}
&=\frac{q^n}{[2l+1]} \times \\ & \quad \times \begin{pmatrix}
q^{-l-\frac{1}{2}} [l+n+\frac{1}{2}]  &
\sqrt{[l+n+\frac{1}{2}][l-n+\frac{1}{2}]} \; \\[8pt]
\sqrt{[l+n+\frac{1}{2}][l-n+\frac{1}{2}]}  &
q^{l+\frac{1}{2}} [l-n+\frac{1}{2}]  
\end{pmatrix}
\begin{pmatrix}
v^{n,\uparrow}_{l,m} \\[2pt] v^{n,\downarrow}_{l,m}
\end{pmatrix}
\\[5pt] &=:
\begin{pmatrix}
P^{11}_{l,m,n} & P^{12}_{l,m,n} \\[2pt] P^{12}_{l,m,n} & P^{22}_{l,m,n}
\end{pmatrix}
\begin{pmatrix}
v^{n,\uparrow}_{l,m} \\[2pt] v^{n,\downarrow}_{l,m}
\end{pmatrix}
\end{align*}
and notice that --- for any $l,m,n$ --- the $2\times 2$ matrix
$$
P_{l,m,n}
=\begin{pmatrix}
P^{11}_{l,m,n} & P^{12}_{l,m,n} \\[2pt] P^{12}_{l,m,n} & P^{22}_{l,m,n}
\end{pmatrix}
$$
is a rank $1$ projection. Thus, since
$(P^{11}_{l,m,n})^2+(P^{12}_{l,m,n})^2=P^{11}_{l,m,n}$,
the matrix
$$
R_{l,m,n}
:=\frac{1}{\sqrt{P^{11}_{l,m,n}}}\begin{pmatrix}
P^{11}_{l,m,n} & \phantom{-}P^{12}_{l,m,n} \\[2pt] P^{12}_{l,m,n} & -P^{11}_{l,m,n}
\end{pmatrix} \;, 
$$
is a rotation. It is, in fact, unipotent, i.e.~$R_{l,m,n}^2=1$.
The vectors
$$
\begin{pmatrix}
w^{n,||}_{l,m} \\[2pt] w^{n,\perp}_{l,m}
\end{pmatrix}
:=R_{l,m,n}
\begin{pmatrix}
v^{n,\uparrow}_{l,m} \\[2pt] v^{n,\downarrow}_{l,m}
\end{pmatrix}
$$
together with $v^{n,\downarrow}_{|n|-\frac{1}{2},m}$ form an
orthonormal basis of $W_n\otimes\C^2$ made of eigenvectors
of the projection $p$, i.e.
$$
p\,w^{n,||}_{l,m}=w^{n,||}_{l,m}\;,\qquad
p\,w^{n,\perp}_{l,m}=0\;.
$$
The space $p(\HH_N^+\otimes\C^2)$ is the span of the vectors:
\begin{align*}
& w^{n,||}_{l,m} &&
\forall\;\;n=-j+1,-j+3,\ldots,j,\;\;
l=|n|+\tfrac{1}{2},|n|+\tfrac{3}{2},\ldots,
\\[2pt]
& v^{n,\downarrow}_{|n|-\frac{1}{2},m} &&
\forall\;\;n=-j+1,-j+3,\ldots,j:n<0\;,
\end{align*}
and for all $m=-l,-l+1,\ldots,l$, with $l=|n|-\tfrac{1}{2}$ in the latter case.\\
Similarly $p(\HH_N^-\otimes\C^2)$  is the span of the vectors:
\begin{align*}
& w^{n,||}_{l,m} &&
\forall\;\;n=-j,-j+2,\ldots,j-1,\;\;
l=|n|+\tfrac{1}{2},|n|+\tfrac{3}{2},\ldots,
\\[2pt]
& v^{n,\downarrow}_{|n|-\frac{1}{2},m} &&
\forall\;\;n=-j,-j+2,\ldots,j-1:n<0\;,
\end{align*}
and for all $m=-l,-l+1,\ldots,l$, with $l=|n|-\tfrac{1}{2}$ in the latter case.

On the vectors $\{w^{n,||}_{l,m}, w^{n,\perp}_{l,m}\}$,  the action of $\mL_E$ and $\mL_F$ will have the form
\begin{align*}
\begin{pmatrix}
\mL_E w^{n,||}_{l,m} \\[2pt] \mL_E w^{n,\perp}_{l,m}
\end{pmatrix} =
\begin{pmatrix}
A_{l,m,n}\; & \ldots \\[2pt]
\ldots & \ldots
\end{pmatrix}
\begin{pmatrix}
w^{n-1,||}_{l,m} \\[2pt] w^{n-1,\perp}_{l,m}
\end{pmatrix} \;,
\end{align*}
and 
\begin{align*}
\begin{pmatrix}
\mL_F w^{n,||}_{l,m} \\[2pt] \mL_F w^{n,\perp}_{l,m}
\end{pmatrix}=
\begin{pmatrix}
B_{l,m,n}\; & \ldots \\[2pt]
\ldots & \ldots
\end{pmatrix}
\begin{pmatrix}
w^{n+1,||}_{l,m} \\[2pt] w^{n+1,\perp}_{l,m}
\end{pmatrix} \;.
\end{align*}
A vector $w^{n,||}_{l,m}$ is in the kernel of $pD^+_j$
if and only if $A_{l,m,n}=0$, and is in the cockernel if
and only if $B_{l,m,n}=0$.
Using the action of $\mL_E$ and $\mL_F$, found out to be given by
\begin{align*}
\mL_Ev^{n,\uparrow}_{l,m} &=\sqrt{[l-n+\tfrac{1}{2}][l+n-\tfrac{1}{2}]}\,v^{n-1,\uparrow}_{l,m} \;,\\
\mL_Ev^{n,\downarrow}_{l,m} &=\sqrt{[l-n+\tfrac{3}{2}][l+n+\tfrac{1}{2}]}\,v^{n-1,\downarrow}_{l,m} \;,\\
\mL_Fv^{n,\uparrow}_{l,m} &=\sqrt{[l-n-\tfrac{1}{2}][l+n+\tfrac{1}{2}]}\,v^{n+1,\uparrow}_{l,m} \;,\\
\mL_Fv^{n,\downarrow}_{l,m} &=\sqrt{[l-n+\tfrac{1}{2}][l+n+\tfrac{3}{2}]}\,v^{n+1,\downarrow}_{l,m} \;,
\end{align*}
a straightforward computation shows that: 
\begin{align*}
B_{l,m,n}=A_{l,m,n+1} &=
\sqrt{ [l-n-\tfrac{1}{2}][l+n+\tfrac{1}{2}] }\, \frac{ P^{11}_{l,m,n}P^{11}_{l,m,n+1} }{ \sqrt{P^{11}_{l,m,n}P^{11}_{l,m,n+1}} }
\\
&\phantom{=} +\sqrt{ [l-n+\tfrac{1}{2}][l+n+\tfrac{3}{2}] }\,
\frac{ P^{12}_{l,m,n}P^{12}_{l,m,n+1} } { \sqrt{P^{11}_{l,m,n}P^{11}_{l,m,n+1}} }
\; .
\end{align*}
Thus if the vector $w^{n,||}_{l,m}$ in the kernel of $pD_j^+$,
the vector $w^{n-1,||}_{l,m}$ is in the cockernel, so that
they give no contribution to the index of $pD_j^+p$,
and the index depends only on the vectors $v^{n,\downarrow}_{|n|-\frac{1}{2},m}$.
From the action above, one finds that for any $n<0$,
$$
\mL_Ev^{n,\downarrow}_{|n|-\frac{1}{2},m}=0 \;,\qquad
\mL_Fv^{n,\downarrow}_{|n|-\frac{1}{2},m}=\sqrt{[-2n]}v^{n+1,\downarrow}_{|n|-\frac{1}{2},m}\neq 0 \;.
$$
The vector \smash[b]{$v^{n+1,\downarrow}_{|n|-\frac{1}{2},m}$} is in the image of $p$
if $n+1<0$ and is in the kernel if $n+1>0$. Thus, all \smash[b]{$v^{n,\downarrow}_{|n|-\frac{1}{2},m}$}
belonging to $p(\HH_j^+\otimes\C^2)$ are in the kernel of $pD^+_jp$
while in the cockernel we have only
\smash[b]{$v^{-1/2,\downarrow}_{0,0}$},
and only in the case it belongs to $p(\HH_j^-\otimes\C^2)$, i.e.~when $j\in 2\N+\frac{1}{2}$. We distinguish then three cases:
\\[10pt]
1.~if $j=\frac{1}{2}$, then $\mathrm{Index}(pD^+_jp)=-1$;
\\[10pt]
2.~if $j=2k+\frac{3}{2}\in 2\N+\frac{3}{2}$, then
\begin{align*}
\mathrm{Index}(pD^+_jp) &=\sum_{n=-j+1,-j+3,\ldots,-1/2}(-2n)  \\
& =\sum_{i=0}^k(4i+1)=(2k+1)(k+1)
=\tfrac{1}{2}(j^2-\tfrac{1}{4})\; ;
\end{align*}
\\[2pt]
3.~if $j=2k+\frac{5}{2}\in 2\N+\frac{5}{2}$, then
\begin{align*}
\mathrm{Index}(pD^+_jp) &=\sum_{n=-j+1,-j+3,\ldots,-3/2}(-2n)-1 =\sum_{i=0}^k(4i+3)-1 \\ & =(2k+1)(k+2) =\tfrac{1}{2}(j^2-\tfrac{9}{4}) \;.
\end{align*}
Note that the equation at point 3.~gives the correct answer also
for $j=\frac{1}{2}$. Since for $k\in\N$, $(2k+1)(k+1)$
and $(2k+1)(k+2)$ are strictly positive, the index is never zero.\cvd
\end{proof}

As anticipated then:
\begin{cor}
Since for different values of $j$ we get different values of 
$\mathrm{Index}(pD^+_jp)$, the
spectral triples $(\OO(\CP^1_q),\HH_j,D_j,\gamma_j,J_j)$ in Prop.~\ref{jest} correspond to distinct K-homology
classes.
\end{cor}

Rational Poincar{\'e} duality allows one to prove several interesting
estimates on the eigenvalues of the twist of the Dirac operator $D$ with a Hermitian finitely
generated projective modules, (cf.~\cite[Thm.~1]{Mos97}).
Let us recall its definition \cite{Con94} for the
particular case of a real spectral triple $(\A,\HH,D,J)$ with $K_1(\A)=0$.
One says that the spectral triple satisfies rational Poincar{\'e} duality if
the pairing $\inner{\,,\,}_D:K_0(\A)\times K_0(\A)\to\Z$, defined by
$$
\inner{[P],[Q]}_D:=\mathrm{Index}(P\otimes JQJ^*)D_j^+(P\otimes JQJ^*) \;,
$$
is non-degenerate, for $P$ a $r\times r$ and $Q$ a $s\times s$ projection. 
Here,
$P\otimes JQJ^*$ is a projection on $\HH\otimes\C^{rs}$
and $D^+_j=D_j|_{\HH_+}\otimes \id_{\C^{rs}}$.

The generators of $K_0(\OO(\CP^1_q))\simeq\Z^2$ can be taken to be the class of the trivial projector $[1]$ 
corresponding to $(1,0)$ and the class of the projector $p$ in  \eqref{eq:P1} corresponding \cite{MNW91,H96} to $(1,1)$. 
Thus a generic element of $K_0(\OO(\CP^1_q))$ can be labelled, with $i,k\in\Z$, as 
$$
(i,k)=(i-k)\,[1]+k[p] \;.
$$

\begin{prop}
The spectral triples $(\OO(\CP^1_q),\HH_j,D_j,\gamma_j,J_j)$ of Prop.~\ref{jest}
satisfy rational Poincar{\'e} duality, for any $j\in\N+\frac{1}{2}$. In particular,
the pairing is given by the explicit formula:
\begin{equation}\label{eq:antilin}
\inner{(i,k),(i',k')}_{D_j}=(ki'-ik')\inner{[p],[1]}_{D_j} ,
\end{equation}
where $\inner{[p],[1]}_{D_j}=\mathrm{Index}(pD^+_jp)$
is the index computed in Prop.~\ref{prop:indexj}.
\end{prop}
\begin{proof}
One repeats the first part of the proof of \cite[Prop.~7.4]{Wag09} 
showing the antisymmetry of the pairing induced by the Dirac operator; 
hence by bilinearity it is always of the form \eqref{eq:antilin}, where $\inner{[p],[1]}_{D_j}=\mathrm{Index}(pD^+_jp)$
is the index computed in Prop.~\ref{prop:indexj}
and is different from zero for all values of $j$.
Since
$$
\inner{(i,k),(k,-i)}_{D_j}=(i^2+k^2)\inner{[p],[1]}_{D_j}
$$
is equal to zero only if $i=k=0$, the pairing is non-degenerate:
for any not zero element $(i,k)$ there is at least another not zero
element $(k,-i)$ such that the pairing of the two in not zero.
This concludes the proof. \cvd
\end{proof}

\section{A digression: calculi and connections}\label{sec:Cal}

\subsection{Covariant differential calculi}
A differential \mbox{$*$-calculus} over a \mbox{$*$-algebra} $\A$ is a differential
graded \mbox{$*$-algebra} $(\Omega^\bullet(\A),\dd)$
with $\Omega^0(\A)=\A$ and $\Omega^{k+1}(\A)=\mathrm{Span}\{a\hspace{1pt}\dd\hspace{1pt}\omega,\,a\in\A,\,\omega\in\Omega^k\}$, for all $k\geq 0$. Requesting a graded Leibniz rule, the differential
is uniquely determined by its restriction to $0$-forms.
The datum $(\mc{M},\delta)$ of an $\A$-bimodule and a real derivation
$\delta:\A\to\mc{M}$, i.e.~such that $\delta(a^*)=\delta(a)^*$,
is called a \emph{first order} \mbox{$*$-calculus}; 
generality is not lost by assuming that $\mc{M}=
\mathrm{Span}\{a\hspace{1pt}\dd\hspace{1pt}b\,,\,a,b\in\A\}$,
as if this is not the case one can replace $\mc{M}$ with the obvious
sub-bimodule. A canonical way to construct a 
differential \mbox{$*$-calculus} from $(\mc{M},\delta)$ is to define
$\Omega^0(\A)=\A$, $\Omega^1(\A)=\mc{M}$ and $\Omega^{k+1}(\A)=
\Omega^k(\A)\otimes_{\A}\Omega^1(\A)$, with product given by the
tensor product over $\A$; the derivation $\delta$ uniquely extends
to a differential $\dd$ giving a differential \mbox{$*$-calculus} on $\A$.
On the other hand, one can take any other graded \mbox{$*$-algebra}
$\Omega^\bullet(\A)$ having $\Omega^0(\A)=\A$ and $\Omega^1(\A)=\mc{M}$,
and (uniquely) extend the derivation using the graded Leibniz rule.

A derivation $\delta_u:\A\to\ker m$, with image the kernel of the
multiplication map $m:\A\otimes\A\to\A$ is given by
$\delta_ua:=a\otimes 1-1\otimes a\,,\,a\in\A$, and for any other
derivation $\delta:\A\to\mc{M}$ there exists a bimodule map
$\jmath:\ker m\to\mc{M}$ such that $\delta=\jmath\circ\delta_u$; in this
sense $\delta_u$ is universal (see e.g.~\cite{CQ95}).

If $\U$ is a Hopf \mbox{$*$-algebra},
one says that the calculus $(\Omega^\bullet(\A),\dd)$ is $\U$-covariant 
if $\Omega^\bullet(\A)$ is a graded
left $\U$-module \mbox{$*$-algebra} (i.e.~the action of $\U$ is a degree
zero map, thus respecting the grading) and $\dd$ commutes with the action of $\U$.
It follows that $\Omega^k(\A)$ is a left
$\A\rtimes\U$-module for any $k\geq 1$.

\subsection{Complex structures}

Suppose $(\Omega^\bullet(M),\dd)$ is the de Rham complex of a smooth manifold $M$.
Note that what here we denote by $\Omega^k$ are \emph{complex valued} $k$-forms.
From an algebraic point of view,
an almost complex structure is a decomposition $\Omega^1(M)=\Omega^{1,0}(M)\oplus
\Omega^{0,1}(M)$ of $1$-forms into a $(1,0)$ and a $(0,1)$ part, and this induces
a corresponding decomposition $\Omega^k(M)=\bigoplus_{r+s=k}\Omega^{r,s}(M)$.
The wedge product of forms
is a bi-graded product, i.e.~$\Omega^{p,q}(M)\wedge\Omega^{r,s}(M)\subset\Omega^{p+r,q+s}(M)$,
and the involution sends $\Omega^{r,s}(M)$ into $\Omega^{s,r}(M)$.
Denoting by $\pi_{r,s}$ the projection $\Omega^{r+s}(M)\to\Omega^{r,s}(M)$,
one can decompose the differential as
$$
\dd|_{\Omega^{p,q}(M)}=\sum_{r+s=p+q+1}\pi_{r,s}\circ\dd|_{\Omega^{p,q}(M)}
=\de+\deb+\ldots \;,
$$
where $\de=\pi_{p+1,q}\circ\dd|_{\Omega^{p,q}}$ has degree $(1,0)$ and
$\bar\de=\pi_{p,q+1}\circ\dd|_{\Omega^{p,q}}$ has degree $(0,1)$.
If $M$ is a \emph{complex manifold}, then $\dd=\de+\deb$ without the additional
terms (in general one may have terms of degree $(2,-1)$, $(-1,2)$, etc.).

From $\dd=\de+\deb$ and $\dd^2=0$ it follows that $\de^2=0$, $\deb^2=0$ and
$\de\deb+\deb\de=0$ (since $\dd^2$ is the sum of the three maps $\de^2$, $\deb^2$ and
$\de\deb+\deb\de$, and they have different degree). In fact, an almost complex manifold
is a complex manifold when one of the following equivalent conditions is satisfied
(\S1.3 of \cite{Wel80}):
\begin{itemize}
\item the Lie bracket of $(1,0)$ vector fields is of type $(1,0)$
(dually to the decomposition of $1$-forms one has the analogous decomposition of
vector fields);
\item $\dd=\de+\deb$;
\item $\deb^2=0$.
\end{itemize}
The second one is what we use to define complex noncommutative spaces.

\begin{df}\label{de:cs}
A complex structure on an algebra $\A$ equipped with a differential \mbox{$*$-calculus} $(\Omega^\bullet(\A),\dd)$ is a bi-graded 
\mbox{$*$-algebra} $\Omega^{\bullet,\bullet}(\A)$
with two linear maps $\de:\Omega^{\bullet,\bullet}(\A)\to\Omega^{\bullet+1,\bullet}(\A)$
and $\deb:\Omega^{\bullet,\bullet}(\A)\to\Omega^{\bullet,\bullet+1}(\A)$
such that $\Omega^k(\A)=\bigoplus_{p+q=k}\Omega^{p,q}(\A)$ and
$\dd=\de+\deb$.
\end{df}
The corresponding Dolbeault complex is the differential complex
\begin{equation}\label{eq:Dolbeault}
\A\stackrel{\deb}{\to}\Omega^{0,1}(\A)
\stackrel{\deb}{\to}\Omega^{0,2}(\A)
\stackrel{\deb}{\to}\ldots
\stackrel{\deb}{\to}\Omega^{0,n}(\A)
\stackrel{\deb}{\to}\ldots \:.
\end{equation}
Note that the condition that $\dd$ is a graded derivation is equivalent to both 
$\de$ and $\deb$ be graded derivations while 
$\dd(a^*)=\dd(a)^*$ is equivalent to $\deb a=\de(a^*)^*$.

The algebra of ``holomorphic elements'',
$\hO(\A) := \ker \left\{ \deb: \A \to \Omega^{(0,1)}(\A) \right\}$, is indeed an algebra over $\C$ by the Leibniz rule. Its elements  
will be referred to, if a bit loosely, as holomorphic functions. 

\subsection{Connections}\label{se:con}
Let $(\Omega^\bullet(\A),\dd)$ be a differential calculus over
a \mbox{$*$-algebra} $\A$ and $\E$ a right $\A$-module.
An ({\emph{affine}) \emph{connection} on $\E$
is a $\C$-linear map $\nabla:\E\to\E\otimes_{\A}\Omega^1(\A)$
satisfying the Leibniz rule, 
\begin{equation}
\nabla(\eta a)=\nabla(\eta)a+\eta\otimes_{\A}\dd a \;,\qquad\forall \; \eta\in\E \ a\in\A, \;.
\end{equation}
By the graded Leibniz rule, any connection is extended uniquely to a $\C$-linear
map $\nabla:\E\otimes_{\A}\Omega^\bullet(\A)\to\E\otimes_{\A}\Omega^{\bullet+1}(\A)$.
Due to the Leibniz rule, the \emph{curvature} $\nabla^2 : \E\to\E\otimes_{\A}\Omega^2(\A)$ 
is right $\A$-linear, $\nabla^2(\eta a)=\nabla^2(\eta)a$, i.e. it is
an element in $\mathrm{Hom}_\A(\E,\E\otimes_{\A}\Omega^2(\A))$.
%\rosso{
%%$\simeq \Omega^2(\A)\otimes_{\A}\mathrm{End}_\A(\E)$
%(a right $\A$-module endomorphism of $\E$ with `coefficients' in $2$-forms).}

Connections for to the universal differential calculus are called
themselves universal. Their importance is twofold: i) a universal connection
on a module $\E$ exists if and only if $\E$ is projective \cite{CQ95};
ii) given a universal connection $\nabla_u$ on $\E$ and a calculus $(\Omega^1(\A),\dd)$,
called $\jmath:\ker m\to\Omega^1(\A)$ the bimodule map intertwining the
differentials --- i.e.~such that $\dd=\jmath\circ\delta_u$ ---,
one constructs a connection $\nabla$ for the latter calculus
using the formula $\nabla:=(\id\otimes \jmath)\circ\nabla_u$.
For $\E=p\A^k$ a finitely generated projective module, this connection (also named the \emph{Grassmannian connection} of $\E$) 
is given by
$$
\nabla_p\eta=p\, \dd \eta \;,
$$
with $\dd$ acting diagonally on $\A^k$, and row-by-column multiplication is understood.
Being the space of all connection an affine space, any other connection differs from $\nabla_p$ by an element in
$\mathrm{Hom}_{\A}(\E,\E\otimes_{\A}\Omega^1(\A))$.  

Affine connections on left modules are defined in a similar manner.

If $\E$ is a bimodule, one defines a \emph{bimodule connection}
as a pair $(\nabla,\sigma)$ of a right module connection
$\nabla:\E\to\E\otimes_{\A}\Omega^1(\A)$ and a bimodule isomorphism
$\sigma:\Omega^1(\A)\otimes_{\A}\E\to\E\otimes_{\A}\Omega^1(\A)$ 
such that $\sigma^{-1}\circ\nabla$ is a left module connection. Explicitly,
this means (cf.~\cite[Sect.~8.5]{Lan02}) there is the \emph{left} Leibniz rule as well:
$$
\nabla(a\eta)=a\nabla(\eta)+\sigma(\dd a\otimes_{\A}\eta)
\;,\qquad\forall\;a\in\A,\eta\in\E \;.
$$
%In turn, the isomorphism $(\sigma\otimes\id)(\id\otimes\sigma):\Omega^1\otimes_{\A}\Omega^1\otimes_{\A}\E
%\to\E\otimes_{\A}\Omega^1\otimes_{\A}\Omega^1$ of  bimodules gives an isomorphism
%$\sigma':\Omega^2\otimes_{\A}\E\to\E\otimes_{\A}\Omega^2$. As a consequence 
%$F_\nabla:=\sigma'^{-1}\circ\nabla^2:\E\to\Omega^2(\A)\otimes_{\A}\E$
%is \segna{right} \rosso{left} $\A$-linear, i.e.~it is an element in $\mathrm{Hom}_\A(\E, \Omega^2(\A) \otimes_{\A} \E)$.
%\rosso{
%%$\simeq \Omega^2(\A)\otimes_{\A}\mathrm{End}_\A(\E)$
%(a right $\A$-module endomorphism of $\E$ with `coefficients' in $2$-forms).}

Given bimodule connections $(\nabla_i,\sigma_i)$ on $\A$-bimodules $\E_i$, $i=1,2$, 
a bimodule connection $(\nabla,\sigma)$ on $\E_1\otimes_{\A}\E_2$ can be defined (cf.~\cite[Prop.~2.12]{KLS10})
by $\nabla=(1\otimes\sigma_2)(\nabla_1\otimes 1)+1\otimes\nabla_2$ and $\sigma=(\id\otimes\sigma_2)(\sigma_1\otimes\id)$.

\subsection{Holomorphic structures on modules} \label{se:holcon}

Given a  complex structure on an algebra $\A$ as in Definition~\ref{de:cs}, 
a \emph{holomorphic connection} on a left $\A$-module $\E$ is simply a connection
$$
\nabla^{\deb}:\E\to\Omega^{0,1}(\A) \otimes_{\A}\E
$$
for the differential calculus $(\Omega^{0,\bullet}(\A),\deb)$.
The connection is called \emph{integrable} or
\emph{flat} if it curvature vanishes: $(\nabla^{\deb})^2=0$. 
In this case, the pair $(\E, \nabla^{\deb})$ is a \emph{holomorphic module} 
(cf.~\cite[Sect.~2]{KLS10}}). Similar definitions are available for right modules.

For an integrable connection, in analogy with \eqref{eq:Dolbeault},
one has a complex
$$
\E
\xrightarrow{ \; \nabla^{\deb}  \; } \Omega^{0,1}(\A) \otimes_{\A}\E 
\xrightarrow{ \; \nabla^{\deb}  \; } \ldots
\xrightarrow{ \; \nabla^{\deb}  \; } \Omega^{0,n}(\A) \otimes_{\A}\E 
\xrightarrow{ \; \nabla^{\deb}  \; } \dots \;.
$$
The zero-th cohomology group of this complex, $H^0(\E, \nabla^{\deb} )$,
will be called the ``space of holomorphic sections'' of $\E$. 
By the Leibniz rule it is, in fact, a (left) module over the algebra $\hO(\A)$ of holomorphic functions.
 
%%% ======================================================================
\section{The complex structure of $\CP^n_q$}\label{sec:5}

A first attempt to classify first order differential calculi
on $\OO(\CP^n_q)$ is in \cite{Wel00}, where it is proven that, for $n\geq 5$, there exists a
unique differential calculus if one requires some (pretty strong)
constraints. One of these is the requirement (cf.~\cite[Sect.~4.2]{Wel00}) 
that $\Omega^1(\CP^n_q)$ is a free left module of rank $n(n+2)$, quite a stong one, given that 
the cotangent bundle of $\CP^n$ is not parallelizable --- the module of sections is not free ---, 
and the rank is $n$ (as a complex vector bundle). 
In that paper there is also a discussion of first order calculi on $\OO(\CP^n_q)$
that are the restriction of calculi on $\OO(S^{2n+1}_q)$.
Few years later, it was proven
in \cite{HK04} that for $\CP^n_q$ there is only one covariant (irreducible, finite-dimensional) first order \mbox{$*$-calculus}.
Higher order differential calculi are studied in \cite{HK06b}. 
As already mentioned, this first order differential calculus
can be realized by commutators with a ``Dirac operator'' \cite{Kra04}. The calculus was \mbox{re-obtained} in
\cite{Bua11a} as the restriction of a distinguished quotient of the bicovariant
calculus on $\OO(SU_q(n+1))$. 
 
We then proceed to complex and related holomorphic structures on $\CP^n_q$. 
This started in \cite{KLS10} for $\CP^1_q$, later generalized to $\CP^2_q$ in \cite{KM11a} and $\CP^n_q$ 
in \cite{KM11b}.

\subsection{The Dolbeault complex}\label{sec:Dol}

For $\CP^n_q$, the differential Dolbeault complex as in \eqref{eq:Dolbeault}
has been constructed in \cite{DD09}. Roughtly speaking, forms
$\Omega^{0,k}(\CP^n_q)$ are given by the equivariant module associated to
the irreducible $*$-representation of $\Uq{su}{n}$ with highest weight
$$
(\,\stackrel{k-1\;\mathrm{times}}{\overbrace{0,\ldots,0}},1,\!\!\!\stackrel{n-k-1\;\mathrm{times}}{\overbrace{0,\ldots,0}}\!\!) \;,
$$
for any $1\leq k\leq n-1$, extended, the representation, to $\Uq{u}{n}$ in a way
that the element $\hat{K} := (K_1 K_2 ... K_n)^{\frac{2}{n+1}}$ (cf.~eq.~(3.1) in \cite{DD09}) is $q^k$  
times the identity.
The module $\Omega^{0,n}(\CP^n_q)$ is simply $\Gamma_{-n-1}$.
A $(0,k)$-form is a vector $\omega=(\omega_{i_1,i_2,\ldots,i_k})$
having components $\omega_{i_1,i_2,\ldots,i_k}\in\OO(SU_q(n+1))$,
with labels satisfying the contraints $1\leq i_1<i_2<\ldots<i_k\leq n$,
and transforming under the $\mc{L}$-action of $\Uq{u}{n}$ according
to the above-mentioned representation. The product of forms is denoted
by $\wprod$ and given by
$$
(\omega\wprod\omega')_{i_1,\ldots,i_{h+k}}=\sum\nolimits_{p\in S^{(h)}_{h+k}}(-q^{-1})^{||p||}\,
\omega_{i_{p(1)},\ldots,i_{p(h)}}\,\omega'_{i_{p(h+1)},\ldots,i_{p(h+k)}}
$$
for all $\omega\in\Omega^{0,h}(\CP^n_q)$ and $\omega'\in\Omega^{0,k}(\CP^n_q)$,
and for all $h,k=0,\ldots,n$ with $h+k\leq n$.
Moreover, we set $v\wprod w:=0$ if $h+k>n$.
Here $S^{(h)}_{h+k}$ are permutations whose inverse
is a $(h,k)$-shuffle, and $||p||$ is the length of the permutation $p$.
The details in \cite{DD09} show that the above is a well-defined associative product. 
Notice that, for any $a\in\OO(\CP^2_q)$ it holds that
$$
\omega\wprod a\, \omega'=\omega \, a\wprod\omega' ,
$$
meaning that the product is a quotient
of the free tensor product over $\OO(\CP^2_q)$.

Any product of $1$-forms, $\omega^i$, $i=1,\ldots,k$,
is also easy to describe:
$$
(\omega^1\wprod \omega^2\wprod\ldots\wprod \omega^k)_{i_1,\ldots,i_k}
=\sum\nolimits_{p\in S_k}(-q^{-1})^{||p||}\omega^1_{i_{p(1)}}
\omega^2_{i_{p(2)}}\ldots \omega^k_{i_{p(k)}} \;,
$$
with $S_k$ the group of permutations of $k$ objects. For $q=1$
this is the antisymmetric tensor product over the algebra.

A graded derivation $\deb: \Omega^{0,k}(\CP^n_q) \to \Omega^{0,k+1}(\CP^n_q)$, of the form 
\begin{equation}\label{ded}
\deb = \sum\nolimits_{1}^n \mL_{\hat{K} X_i} \;,
\end{equation}
for suitable elements $X_i \in \Uq{u}{n+1}$ \cite[eq.~(5.7)]{DD09},
squares to zero, $(\deb)^2=0$, thus giving a covariant (Dolbeault-like) differential calculus $(\Omega^{0,\bullet}(\OO(\CP^n_q)),\deb)$.
Moreover, the map $h \to  \mL_h $ being a $*$-representation, the Hermitian conjugate operator,
$$
\deb^\dag=\sum\nolimits_{1}^n \mL_{ X_i^* \hat{K} } \;,
$$ 
maps $\Omega^{0,k}(\CP^n_q)$ to $\Omega^{0,k-1}(\CP^n_q)$ and squares to zero as well: $(\deb^\dag)^2=0$.
One needs stressing that the above calculus is not a \mbox{$*$-calculus}. 
For later use, we mention that the element $X_i$ above are given by \cite[Lemma~3.13]{DD09}:
\begin{equation}\label{xnm}
X_i :=N_{i,n}M_{i,n}^*
\end{equation}
with $N_{i,n}$ the elements \cite[eq.~(3.7)]{DD09}:
$$
N_{i,n} :=(K_i K_{i+1} \cdots K_n)\, \hat K ^{-1} \;, \qquad i=1, \dots n \,, 
$$
whereas the elements $M_{i,n}$ are defined recursively \cite[eq.~(3.5)]{DD09}) by:
$$
M_{i,n} = E_i M_{i+1,n} - q^{-1} M_{i+1,n} E_i   \;, \qquad i=1, \dots n \,.
$$

\subsection{Hodge decomposition and Dolbeault cohomology}\label{sec:hodge}
We are ready to compute the  cohomology groups $H_{\deb}^\bullet(\CP^n_q)$ of
the complex $(\Omega^{0,\bullet}(\OO(\CP^n_q)),\deb)$
by generalizing the analogue of the Hodge decomposition theorem envisaged
in \cite{DDL08b} for the case $n=2$.
Let 
$$
\Delta_{\deb}=(\deb+\deb^\dag)^2 
$$
be the Hodge Laplacian. We call \emph{harmonic $(0,k)$-forms} the collection:
$$
\mathfrak{H}^{0,k}(\CP^n_q)=\big\{\omega\in\Omega^{0,k}(\CP^n_q)\,\big|\,\Delta_{\deb}
\,\omega=0\big\} \;.
$$
Thus $\omega\in\Omega^{0,k}(\CP^n_q)$ is harmonic if and only if it is in the
kernel of $\deb+\deb^\dag$ ($\ker L=\ker L^2$ for any linear operator $L=L^*$).
Being $(\deb+\deb^\dag)(\omega)$ the sum of two pieces of different degree, both must vanish for
$(\deb+\deb^\dag)(\omega)$ to be zero: hence, $\omega$ is harmonic if and only if $\deb\omega=\deb^\dag\omega=0$. 

\begin{prop}\label{prop:harm}
For all $k$, there is an orthogonal decomposition
\begin{equation}\label{eq:ort}
\Omega^{0,k}(\CP^n_q)=\mathfrak{H}^{0,k}(\CP^n_q)\oplus\deb\Omega^{0,k-1}(\CP^n_q) \oplus\deb^\dag\!\Omega^{0,k+1}(\CP^n_q) \;.
\end{equation}
In particular, there is
exactly one harmonic form for each cohomology class:
$$
H_{\deb}^k(\CP^n_q)\simeq\mathfrak{H}^{0,k}(\CP^n_q) \;.
$$
\end{prop}
\begin{proof}
With the inner product in \eqref{eq:inform}, given two forms $\omega_{1}, \omega_{2}$ of degree $k-1$ and $k+1$ respectively, we have that
$$
\inner{\deb\omega_1,\smash[t]{\deb^\dag\!\omega_2}}=\inner{\smash[t]{\deb^2\omega_1},\omega_2}=0 \;.
$$
Thus, the spaces $\deb\Omega^{0,k-1}(\CP^n_q)$ and $\smash[t]{\deb^\dag\!\Omega^{0,k+1}(\CP^n_q)}$
are orthogonal subspaces of $\Omega^{0,k}(\CP^n_q)$.
It remains to show that a $(0,k)$-form $\eta$ is orthogonal to both
$\deb\Omega^{0,k-1}(\CP^n_q)$ and \smash[t]{$\deb^\dag\!\Omega^{0,k+1}(\CP^n_q)$} if and only if
it is harmonic. This follows from non-degeneracy of the inner product
(i.e.~from the faithfulness of the Haar state).
We have:
$$
\inner{\eta,\deb\omega_1}=\inner{\smash[t]{\deb^\dag}\!\eta,\omega_1}=0 \;, \qquad
\inner{\eta,\smash[t]{\deb^\dag}\!\omega_2}=\inner{\deb\eta,\omega_2}=0 \;,
$$
for all $\omega_1\in\Omega^{0,k-1}(\CP^n_q)$ and $\omega_2\in\Omega^{0,k+1}(\CP^n_q)$
if and only if \smash[t]{$\deb\eta=\deb^\dag\!\eta=0$}, that is
if and only if $\eta$ is harmonic. This establishes the decomposition in \eqref{eq:ort}.

Forms in the subspace $\mathfrak{H}^{0,k}(\CP^n_q)\oplus\deb\Omega^{0,k-1}(\CP^n_q)$
are $\deb$-closed by construction. %by definition of harmonic forms and since $\de^2=0$.
On the other hand,
a $\deb$-closed form \rule{0pt}{10pt}\smash[t]{$\omega\in\deb^\dag\!\Omega^{0,k+1}(\CP^n_q)$} must be harmonic (since
\smash[t]{$(\deb^\dag)^2=0$}), and by orthogonality of the decomposition it must be zero.
It follows that
$$
H_{\deb}^k(\CP^n_q)=\big\{\mathfrak{H}^{0,k}(\CP^n_q)\oplus\deb\Omega^{0,k-1}(\CP^n_q)\big\}
 \big/ \deb\Omega^{0,k-1}(\CP^n_q)=\mathfrak{H}^{0,k}(\CP^n_q) \;,
$$
and this concludes the proof.\cvd
\end{proof}

We now compute $H_{\deb}^k(\CP^n_q)
\simeq\mathfrak{H}^{0,k}(\CP^n_q)=\smash[b]{\ker\Delta_{\deb}\big|_{\Omega^{0,k}(\CP^n_q)}}$. 

\begin{prop}
The Dolbeault cohomology groups of $\CP^n_q$ are given by:
$$
H_{\deb}^0(\CP^n_q)=\C \;,\qquad
H_{\deb}^k(\CP^n_q)=0\quad\forall\;1\leq k\leq n \;.
$$
\end{prop}
\begin{proof}
Lemma 6.3 of \cite{DD09}, for $N=0$ and with $\ell$ replaced by $n$, gives:
$$
\Delta_{\deb}\,\omega=\omega\za\left(\textstyle{\sum_{i=1}^n}q^{-2i}X_iX_i^*+q^{-n-k}[k][n+1]\right) \;,
$$
for any $\omega\in\Omega^{0,k}(\CP^n_q)$, and with $X_i$ the elements making up the operator $\deb$ as in \eqref{ded}.
Since the right hand side is a sum of two positive operators,
$\Delta_{\deb}\,\omega=0$ if and only if one has both $\sum_{i=1}^nq^{-2i}\omega\za X_iX_i^*=0$
and $[k][n+1]=0$. The latter condition implies $H_{\deb}^k(\CP^n_q)=\ker\Delta_{\deb}\big|_{\Omega^{0,k}(\CP^n_q)}=0$
for any $k\neq 0$. 

For the remaining $k=0$ case, Lemma 6.5 of \cite{DD09}, for $N=k=0$, gives:
$$
\za\left(\mathcal{C}_q-\textstyle{\sum_{i=1}^n}q^{n+1-2i}X_iX_i^*\right)\!\big|_{\Omega^{0,0}(\CP^n_q)}
=\frac{q^{-n}+q[n]-[n+1]}{(q-q^{-1})^2}=0
 \;.
$$
Therefore, for any $\omega\in\Omega^{0,0}(\CP^n_q)$ it holds that 
$$
q^{n+1}\Delta_{\deb}\,\omega=\omega\za\mathcal{C}_q=\mathcal{C}_q\az\omega \;,
$$
the second equality following from the fact that
for central elements the left and right canonical actions coincide (cf.~the proof of Lemma 3.1 of \cite{DDL08b}).

Now, as left $\Uq{su}{n+1}$-modules, \cite[Prop.~5.5]{DD09} yields the equivalence:
$$
\Omega^{0,0}(\CP^n_q)\simeq
\bigoplus\nolimits_{m\in\N}V_{(m,0,\ldots,0,m)} \;,
$$
where $V_{(m,0,\ldots,0,m)}$ is the vector space carrying the irreducible representation of highest weight
$(m,0,\ldots,0,m)$.
Finally, from  \cite[Prop.~3.3]{DD09} the restriction of $\mathcal{C}_q$
to this representation is $[m][m+n]$ times the identity operator, and vanishes if and only if $m=0$.
Thus, $H_{\deb}^0(\CP^n_q)=V_{(0,0,\ldots,0)}\simeq\C$ coincides with the
trival representation. This concludes the proof.\cvd
\end{proof}

\subsection{Holomorphic modules}\label{se:cpnhs}
As in \S\ref{se:holcon}, a holomorphic connection on a $\A(\CP^n_q)$-module $\E$ is a connection
associated to the calculus $(\Omega^{0,\bullet}(\CP^n_q),\deb)$.
For the modules $\E=\Gamma_N$ such a connection, 
that here we denote by $\nabla^{\deb}_{\!N}$, was given in \cite[Sect.~6]{DD09}.
Indeed, as discussed in \cite[Sect.~5]{KM11b} these are bimodule connections,
and, using their isomorphism 
$\lambda_N:\Gamma_N\otimes_{\OO(\CP^n_q)}\Omega^{0,1}\to \Omega^{0,1}\otimes_{\OO(\CP^n_q)}\Gamma_N$, 
one passes from the left to the right version.
We need a preliminary lemma.

\begin{lem}\label{lemma:5.3}
For any $\eta\in\Gamma_N$, $\nabla^{\deb}_{\!N}\eta$ is the vector with components
\begin{equation}\label{eq:expconn}
(\nabla^{\deb}_{\!N}\eta)_i=q^{\frac{N}{2}-1}\eta\za F_nF_{n-1}\ldots F_i \;,
\qquad i=1,\ldots,n \;.
\end{equation}
\end{lem}
\begin{proof}
By \cite[Lemma 6.1]{DD09},
the connection \smash[t]{$\nabla^{\deb}_{\!N}$} coincides with the operator $\deb$ on $\Gamma_N$.
In turn, since $\eta\in\Gamma_N$ is by definition in the kernel of the right action of $\Uq{su}{n}$
while $\eta\za K_n^{-1}=q^{N/2}\eta$, the vector $\deb\eta$ has components
\begin{align*}
(\deb\eta)_i &=
\eta\za S^{-1}(\hat{K}X_i)
=\eta\za S^{-1}(\hat{K}N_{i,n}M_{i,n}^*) \\
&=\eta\za S^{-1}(K_iK_{i+1}\ldots K_nM_{i,n}^*)
=q^{-1}\eta\za S^{-1}(M_{i,n}^*K_iK_{i+1}\ldots K_n) \\
&=q^{-1}\eta\za K_i^{-1}K_{i+1}^{-1}\ldots K_n^{-1}S^{-1}(M_{i,n}^*)
=q^{\frac{N}{2}-1}\eta\za S^{-1}(M_{i,n}^*) \;,
\end{align*}
with $X_i=N_{i,n}M_{i,n}^*$ as in \eqref{xnm}. Being $x\mapsto S^{-1}(x^*)$ an algebra morphism,
\begin{align*}
S^{-1}(M_{i,n}^*) &=[S^{-1}(E_i^*),S^{-1}(M_{i+1,n}^*)]_q
=-q[F_i,S^{-1}(M_{i+1,n}^*)]_q \\[2pt] &=
-qF_iS^{-1}(M_{i+1,n}^*)+S^{-1}(M_{i+1,n}^*)F_i \;.
\end{align*}
Since $\eta\za F_i=0$ for all $i\neq n$,
we get: %\eqref{eq:expconn}:
\begin{align*}
(\deb\eta)_i &=q^{\frac{N}{2}-1}\eta\za S^{-1}(M_{i,n}^*)
=q^{\frac{N}{2}-1}\eta\za S^{-1}(M_{i+1,n}^*)F_i \\[2pt]
&=q^{\frac{N}{2}-1}\eta\za S^{-1}(M_{i+2,n}^*)F_{i+1}F_i=\ldots
=q^{\frac{N}{2}-1}\eta\za F_nF_{n-1}\ldots F_i \;.
\end{align*}
This concludes the proof. \cvd
\end{proof}
(For $N=0$, eq.~(13) of \cite{KM11b} has an extra minus sign and misses a factor $q^{-1}$.)

It follows from Lemma 6.1 and Prop.~5.6 of \cite{DD09} that the connection is flat, $(\nabla^{\deb}_{\!N})^2=0$. 
For the corresponding space of holomorphic sections $H^0(\Gamma_N, \nabla^{\deb}_{\!N})$
we have next proposition. This is essentially Cor.~4.2 of \cite{KM11b}, of which we give an 
easier proof (also filling few gaps).
With notations $k_{N,n}={\binom{|N|+n}{n}}$;  
\begin{prop}
The cohomology groups are: $H^0(\Gamma_N, \nabla^{\deb}_{\!N})=0$ if $N>0$, and
$H^0(\Gamma_N, \nabla^{\deb}_{\!N})\simeq\C^{k_{N,n}}$ if $N\leq 0$.
Explicitly, for any $N\leq 0$, $H^0(\Gamma_N, \nabla^{\deb}_{\!N})$ is the $\C$-space of degree
$|N|$ polynomials in the $z_i$'s.
\end{prop}
\begin{proof}
By definition, and using \eqref{eq:expconn}:
$$
H^0(\Gamma_N, \nabla^{\deb}_{\!N})=\big\{\eta\in\Gamma_N:\nabla^{\deb}_{\!N}\eta=0\big\}
=\big\{\eta\in\Gamma_N:\eta\za F_n=0\big\} \;.
$$
Recall that the relations between `coordinates' on $S^{2n+1}_q$ and 
generators of $\OO(SU_q(n+1))$ is $z_i=u^{n+1}_{n+1-i}$, while from \cite[eq.~(4.1)]{DD09}, 
$\pi^{n+1}_i(F_j)=\pi^i_{n+1}(E_j)^*=0$ for all $j=1,\ldots,n$. Hence $z_i\za F_j=
\sum_k\pi^{n+1}_i(F_j)u^j_i=0$ for all indices $i,j$.
Recall also that the elements $\psi^N_{j_0,\ldots,j_n}$ in \eqref{psiN} are a generating family of $\Gamma_N$. For
$N\leq 0$ they are degree $|N|$ monomials in the $z_i$'s, hence
they are in the kernel of $F_n$. Thus,
$$
\psi^N_{j_0,\ldots,j_n}\in H^0(\Gamma_N, \nabla^{\deb}_{\!N}) \qquad \forall\;N\leq 0 \;.
$$
Since for a fixed $N$ the elements
$\psi^N_{j_0,\ldots,j_n}$ are independent over $\C$, and their number is $k_{N,n}=\binom{|N|+n}{n}$,
we also have
$\dim H^0(\Gamma_N, \nabla^{\deb}_{\!N})\geq k_{N,n}$ for all $N\leq 0$.
We show next that $\dim H^0(\Gamma_N, \nabla^{\deb}_{\!N})=k_{N,n}$ if $N\leq 0$ and that 
$\dim H^0(\Gamma_N, \nabla^{\deb}_{\!N})=0$ if $N>0$, thus concluding the proof.

From \cite[ Prop.~5.5]{DD09} we read the decomposition of $\Gamma_N$, 
$$
\Gamma_N=
\Omega^0_N\simeq
\begin{cases}
\bigoplus\nolimits_{m\geq 0}V_{(m+N,0,\ldots,0,m)} \;, &\mathrm{if}\;N> 0\,,\\
\bigoplus\nolimits_{m\geq -N}V_{(m+N,0,\ldots,0,m)} \;, &\mathrm{if}\;N\leq 0\,
\end{cases}
$$
into irreducible representations, with the highest weight vector $v_{m,N}$ of the representation
$V_{(m+N,0,\ldots,0,m)}$ explicitly given by
$$
v_{m,N}:=z_0^m(z_n^*)^{m+N}=
\begin{cases}
(p_{0n})^m\psi^N_{0,\ldots,0,N} &\mathrm{if}\;N>0 \,,\\[2pt]
\psi^N_{-N,0,\ldots,0}(p_{0n})^{m+N} &\mathrm{if}\;N\leq 0 \,.
\end{cases}
$$
Indeed $v_{m,N}\in\Gamma_N$, and using the formul{\ae} for the left action in \eqref{lact} (remembering that 
$z_i=z'_{n+1-i}$), one checks that $E_i\az v_{m,N}=0$, for all $i=1,\ldots,n$, i.e.~$v_{m,N}$ is
a highest weight vector, and $K_i\az v_{m,N}=q^{\frac{1}{2}(m+N)\delta_{i,1}+\frac{1}{2}m\delta_{i,n}}v_{m,N}$,
i.e.~its weight is $(m+N,0,\ldots,0,m)$ as claimed.

Let $T_{m,N}$ be the restriction of $\nabla^{\deb}_{\!N}$ to the subspace $V_{(m+N,0,\ldots,0,m)}$
of $\Gamma_N$. Since left and right canonical actions commute, the image of $T_{m,N}$ is a copy of the same
representation $V_{(m+N,0,\ldots,0,m)}$ inside $\Omega^{0,1}(\CP^n_q)\otimes_{\OO(\CP^n_q)}\Gamma_N$.
For the same reason, $\ker T_{m,N}$ carries a representation of $\Uq{su}{n+1}$.
For fixed $N$, each $T_{m,N}$ has distinct domain and image, hence
$\nabla^{\deb}_{\!N}=\bigoplus_mT_{m,N}$ and
$$
H^0(\Gamma_N, \nabla^{\deb}_{\!N})=\bigoplus\nolimits_m\ker T_{m,N} \;.
$$
Being $\ker T_{m,N}$ a representation of $\Uq{su}{n+1}$, it is either the whole $V_{(m+N,0,\ldots,0,m)}$ or it is $\{0\}$, since the representation $V_{(m+N,0,\ldots,0,m)}$
is irreducible. To discern among the two possibilities, it is enough to check whether or not $v_{m,N}$ is
in the kernel of $T_{m,N}$.
Using $z_i\za F_n=0$ and $z_i\za K_n=q^{\frac{1}{2}}z_i$ one finds
$$
v_{m,N}\za F_n=q^{-\frac{m}{2}}z_0^m\big\{(z_n^*)^{m+N}\za F_n\big\} \;.
$$
Using $z_i\za E_n=u^n_{n+1-i}$, one finds
\begin{align*}
(z_n)^{m+N}\za E_n &=\sum_{k=0}^{m+N-1}q^{\frac{1}{2}(m+N-2k-1)} (z_n)^k u^n_1 (z_n)^{m+N-k-1} \\
&=u^n_1(z_n)^{m+N-1}q^{\frac{1}{2}(m+N-1)}\sum_{k=0}^{m+N-1}q^{-2k} \\
&=u^n_1(z_n)^{m+N-1}q^{-\frac{1}{2}(m+N-1)}[m+N] \;.
\end{align*}
Since $a^*\za F_n=-q^{-1}(a\za E_n)^*$, this finally results into
$$
v_{m,N}\za F_n=-q^{-m-\frac{1}{2}(N+1)}[m+N]
z_0^m(z_n^*)^{m+N-1}(u^n_1)^* \;,
$$
and this is zero if and only if $m+N=0$.

Thus $\ker T_{m,N}\neq\{0\}$ if and only if $m=-N$, admissible only if $N\leq 0$.
If $N\leq 0$, then $H^0(\Gamma_N, \nabla^{\deb}_{\!N})=\ker T_{-N,N}=V_{(0,\ldots,0,-N)}$.
By (3.15) of \cite{DD09} its dimension is given by
\begin{multline*}
\frac{\prod_{1\leq r\leq s<n}(s-r+1)}{\prod_{r=1}^nr!}
\prod_{1\leq r\leq s=n}(s-r+1-N)
\\
=\frac{\prod_{s=1}^{n-1}s!}{\prod_{r=1}^nr!}\frac{(-N+n)!}{(-N)!}
=\binom{-N+n}{n} \;,
\end{multline*}
and this concludes the proof.\cvd
\end{proof}

From last proposition there is a vector space isomorphism
\begin{equation}\label{eq:qhcr}
\bigoplus_{N\leq 0}H^0(\Gamma_N, \nabla^{\deb}_{\!N})\simeq
\frac{\C\inner{z_0,\ldots,z_n}}{\inner{z_iz_j-q^{-1}z_jz_i\,,\,0\leq i<j\leq n}} \;.
\end{equation}
The right hand side is also a ring (in fact, it is a complex unital algebra, although
not a \mbox{$*$-algebra}), called the ``quantum homogeneous coordinate ring''  starting with the paper 
\cite{KLS10} for $n=1$. The isomorphism \eqref{eq:qhcr} becomes
an isomorphism of graded unital algebras if we endow the left hand side
with the product induced by tensor product of bimodules.
From \cite[Prop.~5.2]{KM11b} the product of holomorphic sections is a holomorphic section, a fact  
also inferable from the explicit expression of the isomorphism
$\Gamma_N\otimes_{\Gamma_0}\Gamma_M\to\Gamma_{N+M}$. This we show now, for the sake of completeness.
Recall that $\Gamma_0=\OO(\CP^n_q)$.

\begin{lem}
For any $N,M\in\Z$, it holds that $\Gamma_N\otimes_{\Gamma_0} \Gamma_M\simeq \Gamma_{M+N}$.
\end{lem}
\begin{proof}
It is enough to prove that a) $\Gamma_N\otimes_{\Gamma_0} \Gamma_1\simeq \Gamma_{N+1}$ for all $N\in\mathbb{Z}$,
b) $\Gamma_1\otimes_{\Gamma_0} \Gamma_{-1}\simeq \Gamma_0$. Indeed, from a) and b) it follows that
$\Gamma_{N+1}\otimes_{\Gamma_0} \Gamma_{-1}\simeq \Gamma_N$, and with this one proves that $\Gamma_N\otimes_{\Gamma_0} \Gamma_M\simeq \Gamma_{M+N}$
by induction on $M$. A bimodule map $m:\Gamma_N\otimes_{\Gamma_0} \Gamma_M\to \Gamma_{M+N}$ is given by the multiplication.
We now define bimodule maps $\Gamma_{M+N}\to \Gamma_N\otimes_{\Gamma_0} \Gamma_M$ in the cases a) and b), and prove they are inverse maps to $m$. Define:
\begin{align*}
\phi &: \Gamma_{N+1}\to \Gamma_N\otimes_{\Gamma_0} \Gamma_1 \;,
& \phi(\eta) &=\sum\nolimits_{k=0}^n\eta\, z_k\otimes z_k^* \;,
\\
\chi &: \Gamma_0\to \Gamma_1\otimes_{\Gamma_0} \Gamma_{-1} \;,
& \chi(a) &=\sum\nolimits_{k=0}^nq^{2k} a\, z_k^*\otimes z_k \;.
\end{align*}
A straightforward computation shows that $\sum\nolimits_kz_k\otimes z_k^*$ and $\sum\nolimits_kq^{2k}z_k^*\otimes z_k$
commute with the generators of $\Gamma_0$, so that $\phi$ and $\chi$ are bimodule maps. Moreover:
\begin{align*}
(m\circ\phi)(\eta) &=m\circ\left(\sum\nolimits_k\eta z_k\otimes z_k^*\right)=\eta \sum\nolimits_kz_kz_k^*=\eta 
\;,
\\
(m\circ\chi)(a) &=m\circ\left(\sum\nolimits_kq^{2k}az_k^*\otimes z_k\right)=a\sum\nolimits_kq^{2k}z_k^*z_k=a
\;,
\end{align*}
so $m$ is a left inverse of both $\phi$ and $\chi$ (remember that $\sum\nolimits_kq^{2k}z_k^*z_k=1$).

Being the elements $z_i^*$ a generating family of $\Gamma_1$, and the
elements $z_i$ a generating family of $\Gamma_{-1}$, any $\eta\in \Gamma_N\otimes_{\Gamma_0} \Gamma_1$
can be written as $\eta=\sum_i\eta_i\otimes z_i^*$ with $\eta_i\in \Gamma_N$, and any
$\xi\in \Gamma_1\otimes_{\Gamma_0} \Gamma_{-1}$ can be written as $\xi=\sum_i\xi_i\otimes z_i$
with $\xi_i\in \Gamma_1$. A simple computation yields
\begin{align*}
(\phi\circ m)(\eta) &=\phi\left(  \sum\nolimits_i\eta_iz_i^* \right)=\sum\nolimits_{i,k}\eta_iz_i^*z_k\otimes z_k^*
=\sum\nolimits_{i,k}\eta_i\otimes z_i^*z_kz_k^*\\[5pt] &=\sum\nolimits_i\eta_i\otimes z_i^*=\eta
\;,
\\[5pt]
(\chi\circ m)(\xi) &=\chi\left(  \sum\nolimits_i\xi_iz_i \right)=\sum\nolimits_{i,k}q^{2k}\xi_iz_iz_k^*\otimes z_k
=\sum\nolimits_{i,k}\xi_i\otimes q^{2k}z_iz_k^*z_k\\[5pt] &=\sum\nolimits_i\xi_i\otimes z_i=\xi
\;,
\end{align*}
having used the fact that $z_i^*z_k$ and $z_iz_k^*$ belong to $\Gamma_0$ to move them to the right hand
side of the tensor product.
Thus $m$ is also a right inverse of both $\phi$ and $\chi$. It is then an isomorphism of bimodules.\cvd
\end{proof}

%%% ======================================================================

\subsection{Existence of a twisted positive Hochschild cocycle}\label{sec:exist}
For a closed oriented Riemannian manifold $M$ of real dimension $2n$,
one defines a Hochschild $2n$-cocycle $\tau$,  
the ``fundamental class'' of $M$ \cite[Sect.~VI.2]{Con94}, as 
%as explained in \cite[Sect.~VI.2]{Con94}:
\begin{equation}\label{eq:posone}
\tau(a_0,\ldots,a_{2n}):=\int_M a_0\dd a_1\wedge\dd a_2\wedge\ldots\wedge\dd a_{2n} \;.
\end{equation}
If $M$ is a complex manifold,
a representative of the class $[\tau]$ is:
\begin{equation}\label{eq:postwo}
\int_M a_0\de a_1\wedge\ldots\wedge\de a_n\wedge\deb a_{n+1}\wedge\ldots\wedge\deb a_{2n} \;,
\end{equation}
modulo a proportionality constant that here we neglect. The latter is
a \emph{positive} cocycle in the sense of \cite{CC88}.
It is worth stressing that \eqref{eq:posone} is also cyclic while \eqref{eq:postwo} is not.
Positive representatives of $[\tau]$ form a convex space and,
for $n=1$, there is a bijection between its extreme points and
complex structures on $M$ (cf.~\S{}VI.2 of \cite{Con94}).
Another way to construct positive representatives
of $[\tau]$ uses the Clifford representation of differential
forms, cf.~\cite[Sect.~IV, Example 3]{CC88}, leading to
a cocycle that depends only on the conformal class of the
Riemannian metric (in complex dimension $1$,
conformal and complex structures are equivalent).

For the case of $\CP^n_q$, having the Dolbeault complex, the next step would be to construct a full differential \mbox{$*$-calculus}, reducing to the de Rham complex for $q=1$.
For $n=2$ this was done explicitly in \cite{DL09b}. In \cite[Sect.~6]{KM11b},
a positive representative of their fundamental twisted Hochschild
cocycle was given under the (hidden) assumption that there is a product of forms with the
property that $\Omega^{n,n}(\CP^n_q)\simeq\OO(\CP^n_q)$
is a free bimodule of rank $1$, i.e.~there exists a basis element of $\Omega^{n,n}(\CP^n_q)$
that one would take as a ``volume form''. At an algebraic level (i.e.~without using
operators on Hilbert spaces) such a calculus was given in \cite{HK06b},
where the existence of a volume form was also established. 
One should stress that it is not clear 
whether or not the \mbox{$*$-calculus} in \cite{HK06b} is related to 
the Dolbeault complex discussed here, although it is reasonable to guess that, 
modulo isomorphisms, the former is
an extension of the latter.

We now show that in fact the Dolbeault complex is enough
to define a positive twisted Hochschild $2n$-cocycle (although one would need
a full \mbox{$*$-calculus} in order to have an analogue
of \eqref{eq:posone}). We will give a positive $2k$-cocycle
for any $0\leq k\leq n$. Let us recall some basic facts and definitions.

Let $\A$ be a \mbox{$*$-algebra} and $\eta$ an automorphism of $\A$ (not a $*$-automorphism:
we do not assume $\eta(a^*)=\eta(a)^*$). 
We denote by ${_\eta}\A$ the $\A$-bimodule that is $\A$ itself as a vector
space and as a right module, but has a left module structure `twisted' with $\eta$:
$(a,b)\mapsto\eta(a)b$ for $a,b\in\A$.
The Hochschild cohomology $H\!H_\eta^\bullet(\A)=H^\bullet(\A,{_\eta}\A)$
is the cohomology of the
complex $\big(\mathrm{Hom}_{\C}(\A^\bullet,\C),b_\bullet\big)$, where
the coboundary operators $b$ is \cite{Lod97}:
\begin{multline*}
\hspace{5mm}
b\varphi(a_0,a_1,\ldots,a_k)=
\sum\nolimits_{i=0}^{k-1}(-1)^i
\varphi(a_0,\ldots,a_ia_{i+1},\ldots,a_k) \\[2pt]
+(-1)^k\varphi\bigl(\eta(a_k)a_0,a_1,\ldots,a_{k-1}\bigr) \;.
\hspace{5mm}
\end{multline*}
The cocycle $\varphi$ is called \emph{positive}
if the sesquilinear form on $\A^{n+1}$ given by
$$
\inner{a_0\otimes a_1\otimes\ldots\otimes a_n,
b_0\otimes b_1\otimes\ldots\otimes b_n}_\varphi:=
\varphi\bigl(\eta(b_0^*)a_0,a_1,\ldots,a_n,b_n^*,\ldots,b_1^*\bigr)
$$
is positive semidefinite.

\begin{rem}
\textup{
In the original definition one
assumes that $\inner{\,,\,}_\varphi$ is positive definite. But
already for \eqref{eq:postwo}, for $n\geq 2$, this is not true
since for example elements
$a\otimes a\otimes \ldots\otimes a$
are not zero in $\A^{n+1}$, but $\de a\wedge\de a=0$. Even for $n=1$,
$\inner{\,,\,}_\varphi$ is only positive definite on $\A\otimes (\A/\C)$.
Similarly, looking at the proof of \cite[Thm.~6.1]{KM11b},
it is clear that the cocycle $\varphi$ there is only positive semidefinite,
since the Haar state is faithful, but the map
$\A^{n+1}\to\Omega^{0,n}$,
$a_0\otimes a_1\otimes\ldots\otimes a_n\mapsto a_0\de a_1\ldots\de a_n$,
is not injective.
}\end{rem}

For $\CP^n_q$, one defines non-trivial positive 
twisted Hochschild cocycles with $\eta$ the inverse of the modular automorphism, that is:
$$
\eta(a)=K_{2\rho}^{-1}\az a \;,\qquad\forall\;a\in\OO(\CP^n_q),
$$
with $K_{2\rho}$ the element, implementing the square of the antipode, given in \eqref{eia}.
From \cite[eq.~(3.4)]{DD09} and right invariance of elements of $\OO(\CP^n_q)$,
it follows that the Haar state $h$ is the representative of an element in
$H\!H^0_\eta(\OO(\CP^n_q))$, that is
\begin{equation}\label{eq:modular}
h(ab)=h\big(\eta(b)a\big) \;,\qquad\forall\;a,b\in\OO(\CP^n_q) \;.
\end{equation}
Define $\tau_0:=h$ and, for any $1\leq k\leq n$,
define a $2k$-cochain $\tau_k$ on $\CP^n_q$ by
\begin{align}\label{eq:qcc}
\tau_k(a_0,a_1,\ldots,a_{2k}) &=
\inner{
(\deb a_k^*\wprod\ldots\wprod\deb a_1^*)a_0^*,
\deb a_{k+1}\wprod\ldots\wprod\deb a_{2k}
}_{\Omega^{0,k}(\CP^n_q)}
\notag \\
&\hspace{-1cm}=h\Big(
a_0\big(\deb a_k^*\wprod\ldots\wprod\deb a_1^*\big)^*
\cdot\big(\deb a_{k+1}\wprod\ldots\wprod\deb a_{2k}\big)
\Big) \;.
\end{align}
In the first line we have the canonical inner product of
$\Omega^{0,k}(\CP^n_q)$, as given in \S\ref{sec:3.0}.
Elments $\omega_1,\omega_2\in\Omega^{0,k}(\CP^n_q)$ are
column vectors with $\binom{n}{k}$ components and
with entries in $\OO(SU_q(n+1))$; $\omega_1^*$
is the transposed conjugated row vector, and
the product $\omega_1^*\cdot\omega_2$ in \eqref{eq:qcc}
is the row-by-column product composed with the multiplication
in $\OO(SU_q(n+1))$.

\begin{rem}
\textup{
For $q=1$, $\tau_n$ in \eqref{eq:qcc} coincides with \eqref{eq:postwo} modulo a sign:
the property $(\deb a^*)^*=\de a$ yields
$$\tau_n(a_0,\ldots,a_{2n})=(-1)^{\frac{n(n-1)}{2}}\int_{\CP^n}
a_0\eta_1\eta_2,
$$
where
$\eta_1=\de a_1\wedge\ldots\wedge\de a_n\in\Omega^{n,0}({\CP^n})\simeq\Gamma_{n+1}$
and
$\eta_2=\deb a_{n+1}\wedge \ldots\wedge \deb a_{2n}\in\Omega^{0,n}({\CP^n})\simeq\Gamma_{-n-1}$
are scalar function on $SU(n+1)$, and the product of $(n,0)$ forms with $(0,n)$ forms
is simply the product of the corrisponding functions. The integral is normalized
so that $\int_{\CP^n}1=1$.
}
\end{rem}

\begin{prop}
The map $\tau_k$ in \eqref{eq:qcc} is a positive representative
of an element $[\tau_k]\in H\!H^{2k}_\eta(\OO(\CP^n_q))$.
\end{prop}
\begin{proof}
From \eqref{eq:modular} it follows that
\begin{multline*}
\inner{a_0\otimes a_1\otimes\ldots\otimes a_k,
b_0\otimes b_1\otimes\ldots\otimes b_k}_{\tau_k} \\[4pt]
=\inner{
(\deb a_k^*\wprod\ldots\wprod\deb a_1^*)a_0^*,
(\deb b_k^*\wprod\ldots\wprod\deb b_1^*)b_0^*
}_{\Omega^{0,k}(\CP^n_q)} \;.
\end{multline*}
This is positive semidefinite, since
$\inner{\omega,\omega}_{\Omega^{0,k}(\CP^n_q)}\geq 0$
for all $\omega\in\Omega^{0,k}(\CP^n_q)$.
Let us write
$$
\tau_k(a_0,a_1,\ldots,a_{2k}) =h\Big(
a_0\phi(a_1,\ldots,a_k)\cdot\phi'(a_{k+1},\ldots,a_{2k})
\Big) \;,
$$
with 
$$
\phi(a_1,\ldots,a_k)  :=(\deb a_k^*\wprod\ldots\wprod\deb a_1^*)^* \;, \quad
\phi'(b_{1},\ldots, b_{k})  :=\deb b_{1}\wprod\ldots\wprod\deb b_{k} \;,
$$
and recall that the product between $\phi$ and $\phi'$ is the row-by-column product
between two vectors with entries in $\OO(SU_q(n+1))$.
Using the Leibniz rule and the rule for the involution, that is
$\deb(a_ia_{i+1})^*=a_{i+1}^*(\deb a_i^*)+(\deb a_{i+1}^*)a_i^*$, we compute
\begin{multline}
\sum_{i=1}^k(-1)^i\phi( a_1,\ldots,a_ia_{i+1},\ldots,a_{k+1}) \\ =
(-1)^{1}a_1\phi(a_2,\ldots,a_{k+1})+(-1)^k\phi(a_1,\ldots,a_k)a_{k+1} \;,
\label{eq:secline}
\end{multline}
and
\begin{multline}
\sum_{i=k+1}^{2k}(-1)^i\phi'(a_{k+1},\ldots,a_ia_{i+1},\ldots,a_{2k+1}) \\  =
(-1)^{k+1}a_{k+1}\phi'(a_{k+2},\ldots,a_{2k+1})
+(-1)^{2k}\phi'(a_{k+1},\ldots,a_{2k})a_{2k+1} \;.
\hspace{-3mm}
\label{eq:tirline}
\end{multline}
Therefore,
\begin{align*}
b\tau_k(a_0, \ldots,a_{2k+1}) &= h\Big(a_0a_1\, \phi(\ldots)\phi'(\ldots) \\  
& \phantom{=} -a_0a_1\, \phi(\ldots)\phi'(\ldots)
+(-1)^ka_0\phi(\ldots)\, a_{k+1}\, \phi'(\ldots)
\\ & \phantom{=} +(-1)^{k+1}a_0\phi(\ldots)\, a_{k+1}\, \phi'(\ldots)
+ a_0\phi(\ldots)\phi'(\ldots)\, a_{2k+1}
\\  & \phantom{=} - \eta(a_{2k+1})a_0\, \phi(\ldots)\phi'(\ldots)
\Big) \;.
\end{align*}
having used \eqref{eq:secline} for the second line and \eqref{eq:tirline} for the third line. We can simplify
the first four terms and get:
$$
b\tau_k(a_0,\ldots,a_{2k+1}) =
h\Big(
a_0\phi(\ldots)\phi'(\ldots)a_{2k+1}
-\eta(a_{2k+1})a_0\phi(\ldots)\phi'(\ldots)\Big) \;,
$$
that is zero by \eqref{eq:modular}.\cvd
\end{proof}

%%% ======================================================================

\subsection{Quantum characteristic classes}\label{sec:5.4}
A natural map from equivariant K-theory to equivariant cyclic homology is given in \cite{NT03} (among others), 
and adapted to the present situation in \cite{DL09b}. As explained in Sect.~7.1 of the latter, 
equivariant cyclic homology is paired with twisted Hochschild homology, inducing a pairing, 
$$
H\!H^\bullet_\eta(\OO(\CP^n_q))\times K_0^{\Uq{su}{n+1}}(\OO(\CP^n_q))\to\C \;,
$$
of which here we just give the formula.
For a representation $\sigma:\Uq{su}{n+1}\to M_k(\C)$,
an idempotent $p\in M_k(\OO(\CP^n_q))$ satisfying
\eqref{eq:cov}, and a twisted cycle $\tau\in H\!H^m_\eta(\OO(\CP^n_q))$, one has,
$$
\inner{[\tau],[(p,\sigma)]}=
\tau\Big(
\tr_{\C^k}\bigl(\,
\stackrel{m+1\;\mathrm{times}}{\overbrace{
p\dotimes p\dotimes\ldots\dotimes p }}
\sigma(K_{2\rho}^{-1})^t\bigr)
\Big) \;, 
$$
with $\dotimes$ composition of tensor product over $\C$ with matrix multiplication.

\begin{prop}
For any $N\in\Z$ and any $0\leq k\leq n$, with
$\tau_k$ the cocycle in \eqref{eq:qcc} and $(P'_N,\sigma^N)$ the element in \S\ref{sec:3.2}
one has:
$$
\inner{[\tau_k],[(P'_N,\sigma^N)]}=\begin{cases}
1 & \mathrm{if}\;k=0\;,\\
q^{-n-3}[n][N] & \mathrm{if}\;k=1\;,\\
0 & \mathrm{if}\;k\geq 2\;.
\end{cases}
$$
\end{prop}
\begin{proof}
Recall that $P'_N=R_NP_NR_N^{-1}$, with $R_N=\sigma^N(K_{2\rho})^t$
and $P_N=\Psi_N\Psi_N^\dag$ the projections in \S\ref{sec:3.2}.
Since $R_N$ is a constant matrix and the trace is cyclic:
$$
\inner{[\tau_k],[(P'_N,\sigma^N)]}=
\tau_k\Big(
\tr\bigl(P_N\dotimes P_N\dotimes\ldots\dotimes P_N\sigma(K_{2\rho}^{-1})^t\bigr)
\Big) \;,
$$
so that there is no difference in using $P_N$ or $P_N'$.
Since $P_N^*=P_N$:
\begin{multline*}
\inner{[\tau_k],[(P'_N,\sigma^N)]} \\
=h\Big(
\tr\bigl(P_N\big(\deb P_N\dwprod\ldots\dwprod\deb P_N\big)^*
\big(\deb P_N\dwprod\ldots\dwprod\deb P_N\big)
\sigma(K_{2\rho}^{-1})^t\big)\Big) \;,
\end{multline*}
with $\dwprod$ the composition of $\wprod$ with
matrix multiplication. Using $P_N=\Psi_N\Psi_N^\dag$,
of \cite[eq.~(3.4)]{DD09}, cyclicity of the trace,
$\Psi_N\za K_{2\rho}=\Psi_N\za K_n^{2n}=q^{-nN}\Psi_N$
and $K_{2\rho}\az\Psi_N=\sigma(K_{2\rho})^t\Psi_N$
--- cf.~\eqref{eq:sigmaN} ---
we get
\begin{multline*}
\inner{[\tau_k],[(P'_N,\sigma^N)]} \\
=q^{-nN}h\big(
\Psi_N^\dag
\big(\deb P_N\dwprod\ldots\dwprod\deb P_N\big)^*
\big(\deb P_N\dwprod\ldots\dwprod\deb P_N\big)
\Psi_N
\big)
 \;.
\end{multline*}
From the Leibniz rule and
$\Psi_N^\dag=\Psi_N^\dag P_N$ it follows
$(\deb P_N)P_N=(1-P_N)\deb P_N$.
From this and $\Psi_N=P_N\Psi_N$ it follows
$$
(\deb P_N\dwprod\deb P_N)\Psi_N
=P_N(\deb P_N\dwprod\deb P_N)\Psi_N
=\Psi_N(\nabla^{\deb}_N)^2=0 \;,
$$
since $\nabla^{\deb}_N$ is flat.
Thus, if $k\geq 2$ the pairing
$\inner{[\tau_k],[(P'_N,\sigma^N)]}$
is zero.

\noindent
If $k=0$, $\inner{[\tau_0],[(P'_N,\sigma^N)]}=
h(\Psi_N^\dag\Psi_N)=h(1)=1$.

\noindent
The remaining case is $k=1$. Now,
$$
\inner{[\tau_1],[(P'_N,\sigma^N)]}
=q^{-N}\sum_{i=1}^nh\big(
\Psi_N^\dag(\deb P_N)^*_i(\deb P_N)_i\Psi_N
\big)
 \;,
$$
where $\deb P_N$ is a matrix with entries $(0,1)$-forms,
i.e.~vectors with $n$ components labelled by $i=1,\ldots,n$ .
From \eqref{eq:expconn} (recall that $\deb=\nabla^{\deb}_{\!0}$)
we get:
$$
(\deb P_N)_i=q^{-1}P_N\za F_nF_{n-1}\ldots F_i \;.
$$
The case $N<0$ being similar, let us take $N\geq 0$. Then $\Psi_n^\dag\za F_n=0$, since 
$z_i\za F_n=0$, and $\eta\za F_i=\eta\za E_i=0$ for $\eta\in\Gamma_N$ and $i=1,\ldots,n-1$. Therefore
\begin{align*}
\deb(\Psi_N\Psi_N^\dag) &=
q^{-1}(\Psi_N\za F_nF_{n-1}\ldots F_i)(\Psi_N^\dag\za K_nK_{n-1}\ldots K_i)
\\
&=q^{\frac{N}{2}-1}(\Psi_N\za F_nF_{n-1}\ldots F_i)\Psi_N^\dag \;.
\end{align*}
In turn, all of this yields:
\begin{align*}
\inner{[\tau_1],[(P'_N,\sigma^N)]}
&=q^{-2}\sum_{i=1}^nh\big((\Psi_N\za F_nF_{n-1}\ldots F_i)^*(\Psi_N\za F_nF_{n-1}\ldots F_i)\big) \\
&=q^{-2}\sum_{i=1}^nq^{-2(n-i+1)}h\big((\mL_{F_i\ldots F_n}\Psi_N)^*(\mL_{F_i\ldots F_n}\Psi_N)\big) \;.
\end{align*}
From unitarity of the $\mL$ action:
\begin{align*}
\inner{[\tau_1],[(P'_N,\sigma^N)]}
&=q^{-2}\sum_{i=1}^nq^{-2(n-i+1)}h\big(\Psi_N^\dag(\mL_{F_i\ldots F_n}^*\mL_{F_i\ldots F_n}\Psi_N)\big) \\
&=q^{-2}\sum_{i=1}^nq^{-2(n-i+1)}h\big(\Psi_N^\dag(\mL_{E_n\ldots E_iF_i\ldots F_n}\Psi_N)\big) \\
&=q^{-2}\sum_{i=1}^nq^{-2(n-i+1)}h\big(\Psi_N^\dag(\Psi_N\za F_n\ldots F_iE_i\ldots E_n)\big) \;.
\end{align*}
Since $\Psi_N\za E_i=0$ for $i=1,\ldots,n$ and $[E_i,F_j]=0$ when $i\neq j$, we have:
\begin{align*}
\Psi_N\za F_n\ldots F_i E_i\ldots E_n
&=\Psi_N\za F_n\ldots F_{i+1}[F_i,E_i]E_{i+1}\ldots E_n \\
&=\Psi_N\za F_n\ldots F_{i+1}\frac{K_i^{-2}-K_i^2}{q-q^{-1}}E_{i+1}\ldots E_n \\
&=\Psi_N\za \frac{qK_i^{-2}-q^{-1}K_i^2}{q-q^{-1}}F_n\ldots F_{i+1}E_{i+1}\ldots E_n \;,
\end{align*}
having used in the last equality $K_iF_{i+1}K_i^{-1}=q^{\frac{1}{2}}F_{i+1}$
and $K_iF_jK_i^{-1}=F_j$ for $j>i+1$ (the defining relations of $\Uq{su}{n+1}$). 
From $\Psi_N\za K_i=\Psi_N$, if $i<n$ we get:
$$
\Psi_N\za F_n\ldots F_iE_i\ldots E_n=\Psi_N\za F_n\ldots F_{i+1}E_{i+1}\ldots E_n \;,
$$
and by induction on $i$:
\begin{align*}
\Psi_N\za F_n\ldots F_i E_i\ldots E_n&=\Psi_N\za F_nE_n=\Psi_N\za [F_n,E_n] \\
&=\Psi_N\za \frac{K_n^{-2}-K_n^2}{q-q^{-1}}=[N]\Psi_N \;.
\end{align*}
Therefore,
\begin{align*}
\inner{[\tau_1],[(P'_N,\sigma^N)]}
&=q^{-2}[N]\sum_{i=1}^nq^{-2(n-i+1)}h(\Psi_N^\dag\Psi_N) \\
& 
=q^{-2}[N]\sum_{i=1}^nq^{-2(n-i+1)}
=q^{-n-3}[n][N]
 \;.
\end{align*}
This concludes the proof.\cvd
\end{proof}

As a consequence of previous proposition, the idempotents $P'_N$ represent
distinct elements in equivariant K-theory, since $\inner{[\tau_1],[(P'_N,\sigma^N)]}
=\inner{[\tau_1],[(P'_{M},\sigma^{M})]}$ if and only if $N=M$.
This is consistent with \cite[Prop.~3.8]{Wag09}, where it is shown that these
idempotents generates \smash[t]{$K_0^{\Uq{su}{2}}(\OO(\CP^1_q))$}, that
is an infinite-dimensional free abelian group.

%%% ======================================================================
\section{Monopoles and instantons on $\CP^2_q$}\label{sec:6}

A review of the geometry of $\CP^2_q$ is in \cite{DL08}.
The full \mbox{$*$-calculus} was given in \cite{DL09b}. We now review some results
from \cite{DL09b} and \cite{DL11} on monopoles and instantons as solutions
of anti-self-duality equations.

\subsection{The Hodge star operator on $\CP^2_q$}
On a orientable Riemannian manifold $M$ of (real) dimension $n$, there is
a bimodule isomorphism $\Omega^k(M)\to\Omega^{n-k}(M)$ called the Hodge star operator:
this is an isometry and has square $\pm 1$. It is usually defined
%on real forms,
in local coordinates, using the completely antisymmetric tensor
and the determinant of the metric. With the Hodge star, one defines an
inner product on the space of forms. 

In the noncommutative case (lacking local coordinated), we proceed
in the opposite way: we have a canonical Hermitian structure on forms,
and we use this to define a map $*_H$ that we call ``Hodge star operator''.
We then show that on $\CP^2_q$ this has the correct properties and
the correct $q\to 1$ limit.

The starting point to define $*_H$ is a differential \mbox{$*$-calculus} $(\Omega^\bullet(\A),\dd)$
over a \mbox{$*$-algebra} $\A$.  To have a bimodule isomorphism
$\Omega^k(\A)\to\Omega^{n-k}(\A)$, for some $n$ that we call ``dimension'' of the
calculus, a necessary condition is that $\Omega^n(\A)$ is a free $\A$-bimodule of rank $1$, whose base
element we denote by $\vol$. This is analogue to the condition that the
space is orientable.

We also assume that each $\Omega^k(\A)$, as a right module, has an Hermitian
structure $(\,\cdot, \,\cdot):\Omega^k(\A)\times\Omega^k(\A)\to\A$ 
and is self-dual\footnote{Amongst the many uses of this term, here we mean that the Hermitian structure yields also all homomorphisms of $\Omega^k(\A)$, i.e.~given any right $\A$-module homomorphism $\phi: \Omega^k(\A) \to \A$ there is 
$\eta\in\Omega^k(\A)$ so that $\phi(\cdot)=(\eta, \,\cdot\,)$. }.
Under this assumption, it is possible to prove
there exists a right $\A$-module map $*_H:\Omega^k(\A)\to\Omega^{n-k}(\A)$
uniquely defined by
$$
(*_H\omega_1,\omega_2)\vol=\omega_1^*\, \omega_2
$$
for all $\omega_1\in\Omega^k(\A)$ and $\omega_2\in\Omega^{n-k}(\A)$ (the product
on the right hand side is the product in $\Omega^\bullet(\A)$).
In particular, finitely generated projective modules with the canonical
Hermitian structure are self-dual; in addition, for them it is possible to
prove that the map $*_H$ is also a left $\A$-module map.  
More details on this topic will be reported in \cite{DL11}.

If $n=4$, $\Omega^2(\A)$ is the direct sum of the eigenspaces of $*_H$ corresponding
to the eigenvalues $+1$ and $-1$, called spaces of selfdual, respespectively anti-selfdual
(SD or ASD, for short) $2$-forms.

On $\Omega^2(\CP^2_q)$ the Hodge star operator is given explicitly in \cite{DL09b,DL11} in a way that
we briefly describe. Similarly to the $q=1$ case,
$$
\Omega^{1,1}(\CP^2_q)=\Omega_0^{1,1}(\CP^2_q)\oplus (\Omega_0^{1,1}(\CP^2_q))^\perp
$$
is the (orthogonal) direct sum of a rank $1$ free $\OO(\CP^2_q)$-bimodule
$\Omega_0^{1,1}(\CP^2_q)$ and its orthogonal complement.
A basis element for $\Omega_0^{1,1}(\CP^2_q)$ is given by the
$\Uq{su}{3}$-invariant $2$-form:
$$
\kahl:=\sum\nolimits_{ijk}q^{2i}p_{ij}\dd p_{jk}\wprod\dd p_{ki}=
\sum\nolimits_{ij}q^{2i}\de p_{ij}\wprod\deb p_{ji} \;,
$$
where $p_{ij}=z_i^*z_j$ are the generators of $\OO(\CP^2_q)$
(and we recall that one passes to the notations of \cite{DL09b,DL11}
with the replacement $z_i\to z_{3-i}$). For $q=1$, modulo a
proportionality constant, this is just the K{\"a}hler form
associated to the Fubini-Study metric \cite{DL11}.

There are two possible choices of orientation for $\CP^2_q$, and the corresponding Hodge star operators
differ by a sign. On $\CP^2_q$ with \emph{standard orientation}, a $2$-form is ASD if and only if
it belongs to $(\Omega_0^{1,1}(\CP^2_q))^\perp$
(compare with the classical situation in \cite{Don84}).
On $\CP^2_q$ with \emph{reversed orientation},
that we denote by $\overline{\CP^2_q}$,
a $2$-form is ASD if and only if it belongs to 
$\Omega^{0,2}(\CP^2_q)\oplus\Omega_0^{1,1}(\CP^2_q)\oplus\Omega^{2,0}(\CP^2_q)$;
in particular, the K{\"a}hler form is ASD
(for the classical situation compare with \cite{DK90}).

\subsection{ASD connections and Laplacians}

Using the isomorphism $\Gamma_{-N}\simeq P_N \OO(\CP^2_q)^{k_{N,2}}$ (with $k_{N,2}=\binom{|N|+2}{2}$)
discussed in \S\ref{sec:3.2}, one moves the Grassmannian connection of 
$\E=P_N\OO(\CP^2_q)^{k_{N,2}}$ to $\Gamma_{-N}$.
This yields a connection $\nabla_{\!N}$ given on $\eta\in\Gamma_{-N}$
by
\begin{equation}\label{grconn}
\nabla_{\!N}\eta=\Psi_N^\dag\dd(\Psi_N\eta) \;.
\end{equation}
Its curvature is the operator of left multiplication by the
$2$-form $\nabla_{\!N}^2$ in $\Omega^2(\CP^2_q)$ given by
\begin{equation}\label{grcurv}
\nabla_{\!N}^2=\Psi_N^\dag(\dd P_N)^2\Psi_N \;.
\end{equation}

In \S\ref{sec:3.1} we
saw that Fredholm modules are a good replacement of Chern characters,
as they are used to construct maps $K_0\to\Z$ that are the analogue
of characteristic classes (also called Chern-Connes characters
in K-homology). 

On the other hand on $\CP^2_q$ one can also mimic
the construction of the usual Chern characters by associating to
finitely generated projective modules (sections of noncommutative
vector bundles) suitable integrals of powers of the Grassmannian
connection (in fact of any connection). 
It appears that the correct framework for this is 
equivariant K-theory, as these integrals give numbers (that are not integer valued) 
depending only on the K-theory class of equivariant projective modules.
These maps $K_0^{\U}\to\R$ are described in \cite{DL09b}.

A connection on a bimodule will be called ASD if its curvature
is a right-module endomorphism with coefficients in \emph{anti-selfdual}
$2$-forms.

In \cite{DL09b} we studied $U(1)$-monopoles on $\CP^2_q$, i.e.~ASD connections on
the (line bundle) modules $\Gamma_{-N}$.
The connection $\nabla_{\!N}$ on $\Gamma_{-N}$ is the one 
in \eqref{grconn}.  The corresponding curvature 
$\nabla_{\!N}^2$, as in \eqref{grcurv}, is a
scalar $2$-form since right-module endomorphisms are given by $\mathrm{End}_{\OO(\CP^2_q)}(\Gamma_{-N})\simeq\OO(\CP^2_q)$.
We showed that $\nabla_{\!N}$ is left $\Uq{su}{3}$-invariant, i.e.~it
commutes with the left action of $\Uq{su}{3}$. From this, it follows
that the curvature is an invariant $2$-form, and then it is ASD
on $\overline{\CP^2_q}$. Explicitly, one has
$$
\nabla_{\!N}^2=q^{N-1}[N]\,\kahl \;.
$$
In \cite{DL11} we are continuing the project and describe $SU_q(2)$ one-instantons on $\CP^2_q$, i.e.~ASD connections on a `rank $2$ homogeneous
vector bundle' with first Chern number equal to $0$ and second Chern number equal to $1$. Following Donaldson \cite[Example 4.1.2]{DK90} we choose the reverse orientation on $\CP_q^2$. The ASD condition can be reformulated as a system of
finite-difference equations (differential equations for $q=1$, while 
derivatives are replaced by $q$-derivatives  when $q\neq 1$), and provide a family of solutions
`parametrized' by a non-commutative space that is a cone over $\CP^2_q$. 

Given the monopole connection $\nabla_{\!N}$ on $\Gamma_{-N}$, one can also 
define the associated Laplacian $\Delta_N:=(\nabla_{\!N})^*\nabla_{\!N}$,
where $(\nabla_{\!N})^*$ 
is the adjoint of $\nabla_{\!N}$.
The eigenvalues $\{\lambda_{k,N}\}_{k\in\N}$ of $\Delta_N$, explicitly computed in \cite{DL09b}, are given by:
\begin{align*}
\lambda_{k,N}&=
(1+q^{-3})[k][k+N+2]+[2][N] &&\mathrm{if}\;N\geq 0\;,\\
\lambda_{k,N}&=
(1+q^{-3})[k+2][k-N]+[2][N] &&\mathrm{if}\;N<0\;.
\end{align*}
We point out that for $q=1$,
$\lambda_{k,N}=2(k^2+kN+2k+N)=\lambda_{k,-N}$ for any $N\geq 0$.
On the other hand for $q\neq 1$,
the spectrum of $\Delta_N$ is not symmetric under
the exchange $N \leftrightarrow -N$;
%not even when exchanging in addition
%the parameter as $q \leftrightarrow q^{-1}$:
the quantization removes some degeneracies.
A similar phenomenon was observed in \cite{LRZ07}
for $\CP^1_q$. There is a simple relation, 
$$
\lambda_{k,N}-\lambda_{k,-N}=(1-q^{-3})[2][N] \;, \qquad \textup{for all} \quad N\geq 0 \;.
$$

\appendix

\section{On Chern characters and Fredholm modules}\label{app}

In Prop.~\eqref{q-ind} we gave maps
$$
\varphi_k:=\inner{[F_k],\,.\,}:\;K_0(\OO(\CP^n_q))\to\Z \;,
$$
that, when $K_0(\OO(\CP^n_q))$ is identified with $\Z^{n+1}$ using the generators
$[P_0]$, $[P_{-1}],..., [P_{-n}]$, are morphisms of abelian groups $\Z^{n+1}\to\Z$.

For $q=1$, using the embeddings $\imath:\CP^k\to\CP^n$ one has has maps
$$
\mathrm{Ch}_k:K^0(\CP^n)\to \Q \;,\qquad
\mathrm{Ch}_k(\mathcal{V})=\int_{\CP^k}\imath^*\mathrm{ch}_k(\mathcal{V}) \;,
$$
where $\mathcal{V}\to\CP^n$ is a vector bundle, and $\mathrm{ch}_k(\mathcal{V})$
its $k$-th Chern character.
Similarly to above, one can identify $K^0(\CP^n)$
with $\Z^{n+1}$ using corresponding line bundles $L_0,L_{-1},\ldots,L_{-n}$,
where $L_0=\CP^n\times\C$ is the trivial line bundle, $L_{-1}\to\CP^n$ is the dual of the tautological bundle
and $L_{-N}
=(L_{-1})^{\otimes n}$. We compare the maps $\varphi_k$
and $\mathrm{Ch}_k$ as morphisms of abelian groups $\Z^{n+1}\to\Q$.
From Prop.~\eqref{q-ind} we know $\varphi_k(P_{-N})=\tbinom{N}{k}$. 
We need to compute $\mathrm{Ch}_k(L_{-N})$. 

For a line bundle $L$, the total Chern character is $\mathrm{ch}(L)=e^{c_1(L)}$,
being the first Chern class $c_1(L)$ the only non-zero such a class for a line bundle.
Since $\mathrm{ch}(L\otimes L')=\mathrm{ch}(L)\mathrm{ch}(L')$, we have
$\mathrm{ch}(L_{-N})=\mathrm{ch}(L_{-1})^N=e^{Nc_1(L_{-1})}$
and $\mathrm{ch}_k(L_{-N})=\frac{N^k}{k!}c_1(L_{-1})^k$.
By \cite[Lemma 2.3.1]{Gil95}, $x:=\imath^*c_1(L_{-1})$ is exactly the
first Chern number of the analogous bundle $L_{-1}$ on $\CP^k$, and
the integral is normalized such that
$\int_{\CP^k}x^k=1$. Therefore:
$$
\mathrm{Ch}_k(L_{-N})=\frac{1}{k!}N^k=\frac{1}{k!}\sum_{j=0}^k\,j!\,{k\brace j}
\binom{N}{j}
 \;,
$$
where ${k\brace j}$ are the Stirling numbers of the second kind \cite{GKP94}. Hence
$$
\mathrm{Ch}_k=\frac{1}{k!}\sum_{j=0}^k{k\brace j}j!\,\varphi_j \;,
$$
as maps $\Z^{n+1}\to \Q$. In particular,
$$
\mathrm{Ch}_0=\varphi_0 \;,\qquad
\mathrm{Ch}_1=\varphi_1 \;,\qquad
\mathrm{Ch}_2=\varphi_2+\tfrac{1}{2}\varphi_1 \;,
$$
with their inverses: $\varphi_0=\mathrm{Ch}_0$, $\varphi_1=\mathrm{Ch}_1$ 
and $\varphi_2=\mathrm{Ch}_2-\frac{1}{2}\mathrm{Ch}_1$, the latter combination always being integer valued. These could be named the `rank', `monopole number' and `instanton number' of the bundle, respectively.

%%% ======================================================================

\end{document}